\documentclass[11pt,reqno,oneside]{amsart}

\usepackage{graphicx}
\usepackage{amssymb}
\usepackage{epstopdf}
\usepackage[a4paper, total={5.56in, 9in}]{geometry}
\usepackage{array} 
\usepackage{hyperref}
\usepackage{bm}
\hypersetup{hypertex=true,colorlinks=true,linkcolor=blue,anchorcolor=blue,citecolor=blue}
\usepackage{amsmath,amsfonts,amsthm,mathrsfs,amssymb,cite}
\usepackage[usenames]{color}
\usepackage{xcolor}

\newtheorem{theorem}{Theorem}[section]
\newtheorem{corollary}[theorem]{Corollary}
\newtheorem{lemma}{Lemma}[section]
\newtheorem{proposition}{Proposition}[section]
\theoremstyle{definition}
\newtheorem{definition}{Definition}[section]

\theoremstyle{remark}
\newtheorem{remark}{Remark}[section]
\numberwithin{equation}{section}
\numberwithin{equation}{section}

\theoremstyle{remark}
\usepackage{cite}
\usepackage{enumerate}
\allowdisplaybreaks
\DeclareMathOperator{\grad}{\nabla}

\DeclareMathOperator{\supp}{supp}

\def\dif{\mathop{}\hphantom{\mskip-\thinmuskip}\mathrm{d}}
\let\daccent\d
\let\d\relax
\newcommand\d{\ifmmode\dif\else\expandafter\daccent\fi}

\title[Non-radiating elastic sources at corners]{Non-radiating elastic sources in inhomogeneous elastic media at corners with applications}

\author{Huaian Diao}
\address{School of Mathematics and Key Laboratory of Symbolic Computation and Knowledge Engineering of Ministry of Education, Jilin University, Changchun, China}
\email{diao@jlu.edu.cn, hadiao@gmail.com}

\author{Yueran Geng}
\address{School of Mathematics, Jilin University, Changchun 130012, China\vspace*{-2mm}}
\address{and\vspace*{-2mm}}
\address{Department of Mathematics, City University of Hong Kong, Kowloon, Hong Kong SAR, China}
\email{gengyr23@mails.jlu.edu.cn}

\author{Ruixiang Tang}
\address{School of Mathematics, Jilin University, Changchun 130012, China}
\email{tangrx23@mails.jlu.edu.cn}

\begin{document}
	\maketitle

	\begin{abstract}
		This paper is concerned with non-radiating elastic sources in inhomogeneous elastic media. We demonstrate that the value of non-radiating elastic sources must vanish at convex corners of their support, provided the sources exhibit H\"older continuous regularity near the corner. Additionally, their gradient must satisfy intricate algebraic relationships with the angles defining the underlying corners, assuming the sources have $C^{1,\alpha}$ regularity  with $\alpha\in (0,1)$ in the neighborhood of the corners. Our analysis employs complex geometrical optics (CGO) solutions as test functions within a partial differential system to conduct \textcolor{black}{asymptotic analysis} near the corners. These characterizations enable us to establish unique identifiability results for determining the position and shape of radiating elastic sources from a single far-field measurement, both locally and globally. The uniqueness of such identification is a longstanding challenge in inverse scattering with a rich history. Specifically, when the support of a radiating elastic source is a convex polygon and the source is H\"older continuous at the corners, we can simultaneously determine the source's shape and its values at the corners. Furthermore, when the source function exhibits $C^{1,\alpha}$ regularity in the neighborhood of a corner, the gradient at that corner can typically be determined. \textcolor{black}{Additionally, when the support includes a convex sectorial corner and the elastic source satisfies certain generic conditions, we demonstrate that such a source must radiate at any frequency.}

		\noindent{\bf Keywords:}~~non-radiating sources, inverse elastic source problems, a single far-field measurement, unique identifiability, \textcolor{black}{CGO solutions}

		\noindent{\bf 2020 Mathematics Subject Classification:}~~35Q74, 35R30, 74B05, 86A22
	\end{abstract}

\section{Introduction} \label{section_introduction}


\subsection{Mathematical setup} 

Let $ \lambda $ and $ \mu $ be the Lam\'e constants defined in $\mathbb R^2$ satisfying the strong convexity condition 
\begin{align} \label{ineq_introduction_strong_convexity_condition}
	\mu > 0, \quad \lambda + \mu > 0.
\end{align}
The constants $ \lambda $ and $ \mu $ are referred to as the compression modulus and the shear modulus respectively. The positive function $ \rho\left(\mathbf{x}\right)\in C^2\left(\mathbb{R}^2\right) $ represents the density of the inhomogeneous elastic medium such that $ \rho\left(\mathbf{x}\right)=1 $ after normization for $ \mathbf{x}\in\mathbb{R}^2 \backslash \overline{D_R} $, where  $D_R:=\left\{\mathbf{x} \in \mathbb{R}^2:\left| \mathbf{x} \right| < R\right\} $ for some $R\in \mathbb R_+$.  Let $ \mathbf{u}\left(\mathbf{x}\right)\in\mathbb{C}^2 $ describe the displacement vector field generated by the external force $ \mathbf{f}\left(\mathbf{x}\right)\in L^{\infty}\left(\mathbb{R}^2\right) $, where   $\mathbf{f}\left(\mathbf{x}\right)$  is assumed to possess a compact support $ \overline{\Omega} $ such that $ \Omega\subset D_R$. Here  $\Omega$ is a bounded Lipschitz domain with a connected complement. Then the time-harmonic Navier equation in an inhomogeneous medium is
\begin{align} \label{model_global_non-nested_R2}
	\mathcal{L} \mathbf{u}+\omega^2 \rho\left(\mathbf{x}\right) \mathbf{u}=\mathbf{f}\left(\mathbf{x}\right) \quad\text{in } \mathbb{R}^2,  
\end{align}
where $ \mathcal{L} := \mu \Delta + \left(\lambda+\mu\right) \nabla \nabla \cdot$ and  $ \omega > 0 $ denotes the angular frequency. 

With the help of the Helmholtz decomposition in $ \mathbb{R}^2\backslash \overline{D_R} $, the displacement vector field $ \mathbf{u} $ can be decomposed as
\begin{align*} 
	\mathbf{u}=\mathbf{u}_{\mathrm{p}}+\mathbf{u}_{\mathrm{s}},
\end{align*}
where $ \mathbf{u}_{\mathrm{p}} $ and $ \mathbf{u}_{\mathrm{s}} $ are called the compressional and shear part of the displacement $ \mathbf{u} $ respectively. It follows that
\begin{align*}
	\mathbf{u}_{\mathrm{p}} = - \frac{1}{k_{\mathrm{p}}^2} \nabla \nabla \cdot \mathbf{u}, \quad \mathbf{u}_{\mathrm{s}} = \frac{1}{k_{\mathrm{s}}^2} \textbf{curl} \operatorname{curl} \mathbf{u} ,
\end{align*}
where $ k_{\mathrm{p}} $ and $ k_{\mathrm{s}} $ are the compressional and shear wave numbers of the homogeneous and isotropic background medium respectively, which are given by 
\begin{align} \label{eq_introduction_kpks} 
	k_{\mathrm{p}} = \frac{\omega}{\sqrt{\lambda + 2\mu}}, \quad k_{\mathrm{s}} = \frac{\omega}{\sqrt{\mu}}.
\end{align}
The notations $ \operatorname{curl} $ and $ \textbf{curl} $ represent 
\begin{align*}
	\operatorname{curl}\ \mathbf{u} = \frac{\partial u_2}{\partial x_1} - \frac{\partial u_1}{\partial x_2}, \quad \textbf{curl}\ u = \left( \frac{\partial u}{\partial x_2}, -\frac{\partial u}{\partial x_1} \right)^{\top}, 
\end{align*}
for $ \mathbf{u} = \left(u_1, u_2\right)^{\top}, \mathbf{x} = \left(x_1, x_2\right)^{\top} $, and the superscript $ \top $ denotes the transpose. 

To guarantee that the direct scattering problem is well-posed, the Kupradze radiation condition is posed as follows
\begin{align} \label{Kupradze_radiation_condition}
	\lim\limits_{r\to \infty} r^{\frac{1}{2}} \left(\frac{\partial\mathbf{u}_{\alpha}}{\partial r} - \mathrm{i} k_\alpha \mathbf{u}_{\alpha}\right)=0, \quad r:=\vert \mathbf{x}\vert, \ \alpha=\mathrm{p},\mathrm{s},
\end{align}
uniformly in all directions $ \hat{\mathbf{x}} :=\mathbf{x}/\left|\mathbf{x}\right| \in \mathbb{S}^1 := \left\{\mathbf{x}= \left(x_1, x_2\right)^{\top} \in \mathbb{R}^2 : x_1^2 + x_2^2 = 1 \right\} $, where $ \mathrm{i} := \sqrt{-1} $ and $ \partial/\partial r $ is the derivative along the radial direction from the origin. Moreover, this radiation condition implies that, as $ r = \left| \mathbf{x} \right| \to \infty $, $ \mathbf{u} $ has the following asymptotic expansion
\begin{align*} 
	\mathbf{u}\left(\mathbf{x}\right) = \frac{e^{\mathrm{i}k_{\mathrm{p}}r}}{r^{\frac{1}{2}}} \mathbf{u}_{\mathrm{p}}^{\infty} \left(\hat{\mathbf{x}}\right) + \frac{e^{\mathrm{i}k_{\mathrm{s}}r}}{r^{\frac{1}{2}}} \mathbf{u}_{\mathrm{s}}^{\infty} \left(\hat{\mathbf{x}}\right) + \mathcal{O} \left(r^{-\frac{3}{2}}\right) ,
\end{align*}
where $ \mathbf{u}_{\mathrm{p}}^{\infty} $ and $ \mathbf{u}_{\mathrm{s}}^{\infty} $ are denoted as the far field patterns of $ \mathbf{u}_{\mathrm{p}} $ and $ \mathbf{u}_{\mathrm{s}} $, respectively. We define the far field pattern $ \mathbf{u}^{\infty} $ of the scattered field $ \mathbf{u} $ as
\begin{align} \label{def_introduction_uinfty}
	\mathbf{u}^{\infty}\left(\hat{\mathbf{x}}\right) := \mathbf{u}_{\mathrm{p}}^{\infty}\left(\hat{\mathbf{x}}\right) + \mathbf{u}_{\mathrm{s}}^{\infty}\left(\hat{\mathbf{x}}\right), \quad \hat{\mathbf{x}} \in \mathbb{S}^1 .
\end{align}
Then the following relations hold
\begin{align*}
	\mathbf{u}_{\mathrm{p}}^{\infty}\left(\hat{\mathbf{x}}\right) = \left(\mathbf{u}^{\infty}\left(\hat{\mathbf{x}}\right)\cdot\hat{\mathbf{x}}\right)\hat{\mathbf{x}}, \quad \mathbf{u}_{\mathrm{s}}^{\infty}\left(\hat{\mathbf{x}}\right) = \left(\hat{\mathbf{x}}^{\perp}\cdot\mathbf{u}^{\infty}\left(\hat{\mathbf{x}}\right)\right)\hat{\mathbf{x}}^{\perp}.
\end{align*}
Moreover, $ \mathbf{u}\left(\mathbf{x}\right) $ and $ \mathbf{u}^{\infty}\left(\hat{\mathbf{x}}\right) $ \textcolor{black}{are in one-to-one correspondence} by the Rellich theorem \cite{hahner1998acoustic}. 

In this work, we are concerned with the inverse source problem formulated as 
\begin{align} \label{problem_introduction_u_fnablafOmega}
	\mathcal{F}\left(\Omega; \mathbf{f}\right) = \mathbf{u}^{\infty} \left(\hat{\mathbf{x}};\omega\right),
\end{align}
where $ \mathcal{F} $ is implicitly defined by the scattering system \eqref{model_global_non-nested_R2} and \eqref{Kupradze_radiation_condition}. Our focus is on recovering the shape and position of the support $ \Omega $ of  $ \mathbf{f} $ by a single far-field measurement. Specifically, if a far-field pattern $ \mathbf{u}^{\infty} \left(\hat{\mathbf{x}};\omega\right) $ is collected for all $ \hat{\mathbf{x}}\in\mathbb{S}^1 $ and a single fixed angular frequency $ \omega $, it is referred to as a single far-field measurement for the inverse problem \eqref{problem_introduction_u_fnablafOmega}. If measurements are taken over multiple frequencies, the problem is referred to as involving multiple measurements.  

For a given frequency  $ \omega $, the source term $ \mathbf{f} $ in \eqref{model_global_non-nested_R2} is said to be non-radiating if the  corresponding far-field pattern $ \mathbf{u}^{\infty}\left(\hat{\mathbf{x}}\right) $ vanishes on $\mathbb S^1$. This implies that $\mathbf u \equiv \mathbf 0$ in $\mathbb R^2  \setminus \overline \Omega$ by the Rellich theorem \cite{hahner1998acoustic}.  In other words, when $\mathbf f$  is non-radiating, it is invisible to far-field measurements and cannot be detected. Conversely, if $\mathbf f$ is radiating, the far-field measurement is non-zero, implying that $\mathbf f$ is observable. In this case, it is possible to recover the geometric information of the support  $\Omega$ and the physical configuration of $\mathbf f$ by far-field measurements. 

Throughout this paper, we assume that $ \omega\in \mathbb R_+ $ is an arbitrarily fixed angular frequency. The well-posedness of the direct problem \eqref{model_global_non-nested_R2} and \eqref{Kupradze_radiation_condition} can be readily obtained by Dirichlet-to-Neumann (DtN) map and variational approach. The main result that we aim to establish is the characterization of the non-radiating sources in inhomogeneous elastic media at corners. The detailed characterization is provided in Theorem \ref{thm_inverse_problem_non-radiating}. When the elastic source $\mathbf f$ is radiating, the non-radiating source characterization at corners (as established in Theorem \ref{thm_inverse_problem_non-radiating}) allows us to uniquely determine the shape and position of the support  $\Omega$ of $\mathbf f$ from a single far-field measurement. The support $\Omega$ can be a convex polygon. Once the shape and position of $\Omega$ is uniquely determined, we further show that the radiating source $ \mathbf{f} $  at the corner points can also be uniquely recovered by a single far-field measurement. Moreover,  the gradient $ \nabla\mathbf{f} $ of the radiating source $ \mathbf{f} $ at the corners can fulfill certain algebraic relationships with the angles of the underlying corners. These results are detailed in Theorem \ref{thm_inverse_problem_global_uniqueness}.

\subsection{Connection to previous results and main findings}

\textcolor{black}{The Navier equation \eqref{model_global_non-nested_R2}, as presented in this paper, is formulated in two dimensions. Its application is driven by the need to model elastic deformations in thin or planar structures, such as plates or cross-sections, offering a simplified alternative to three-dimensional elasticity. This approach enhances computational efficiency while preserving essential physical characteristics, including stress and strain interactions \cite{reddy2006theory, virieux2009overview}.} The inverse source problem is a fundamental topic in inverse scattering theory with a wide range of applications, including medical imaging \cite{hashmi2020optical}, photoacoustic tomography \cite{arridge1999optical}, and seismic monitoring \cite{alves2014identification, ammari2013time, song2011full}. 

It is well-established that the general source function cannot be uniquely determined from a single far-field measurement, as this inverse problem is underdetermined. For instance, consider the source function $ \mathbf{f}_0 := \left(\mathcal{L} + \omega^2 \rho(\mathbf{x})\right)\mathbf{g} $, where $ \mathbf{g} \in C^{\infty}_0 \left(\mathbb{R}^2\right) $. It follows that $ \mathbf{f} $ and $ \mathbf{f} + \mathbf{f}_0 $ generate identical far-field measurements. The inverse problem associated with a single far-field measurement, as given by \eqref{problem_introduction_u_fnablafOmega}, is essentially a manifestation of Schiffer's problem \cite{bleistein1977nonuniqueness, colton2018looking, colton2019inverse}, which concerns the unique identification of the position and shape of a scatterer from a single far-field measurement. Significant progress has been made in the context of Schiffer’s problem for inverse acoustic, elastic, and electromagnetic scattering by introducing a priori information about the scatterer, such as constraints on its size or geometry \cite{colton2018looking, colton2019inverse}, among others.

Extensive research has focused on characterizing non-radiating sources and establishing uniqueness results for the inverse source problem. In \cite{blaasten2018nonradiating}, it is shown that if the support of a source contains a non-degenerate concave or convex corner in $\mathbb{R}^2$, or an edge corner in $\mathbb{R}^3$, a non-radiating acoustic source in a homogeneous medium must vanish at these corner points, provided the source is H\"older continuous in a neighborhood of each corner. In \cite{hu2020inverse}, for an acoustic source in an inhomogeneous background medium, it is demonstrated that if the support of the source is a convex polygon, then both the support and the zeroth and first-order derivatives of the source function at corner points can be uniquely determined from a single far-field measurement. Similar results for non-radiating electromagnetic sources are found in \cite{blaasten2021electromagnetic, li2024nonradiating}. The authors in \cite{DFLW2025} characterize inclusions with corners and semilinear terms by minimal measurements. Furthermore, effective electromagnetic scattering problems involving embedded obstacles is examined in \cite{DLMW20251}. Corresponding developments in the study of scattering phenomena for medium scattering problems can be found in \cite{blaasten2021scattering,blaasten2014corners,cakoni2023regularity,cakoni2023singularities,cakoni2020corner,DFLY24,DTLT24, DLMW20252,elschner2018acoustic,kow2024scattering,paivarinta2017strictly,salo2021free}. 

In the case of elastic source scattering in a homogeneous elastic medium, \cite{blaasten2018radiating} shows that if a non-radiating external force has a non-degenerate concave or convex corner in $\mathbb{R}^2$, or an edge corner in $\mathbb{R}^3$, then this force must vanish at the corner. This result is used to determine the convex polyhedral shape of the source as well as the source value at the corner, all from a single far-field measurement. In \cite{kow2021characterization}, the authors characterize non-radiating volume and surface sources for elastic waves in anisotropic inhomogeneous media in $\mathbb{R}^3$ and show that the radiating elastic source can be uniquely determined from an explicit formula involving near-field measurements. The non-radiating volume elastic source is orthogonal to the completion of the solution set for the Navier equation in an anisotropic inhomogeneous medium, without external volume sources. Notably, the anisotropic inhomogeneous medium $D$ considered in \cite{kow2021characterization} is characterized solely by the elasticity tensor $\mathbf{C}$, which is compactly supported by $D$, differing from a homogeneous elastic medium. However, the density of $D$ does not exhibit a jump singularity, unlike the counterpart in a homogeneous elastic background medium.

In \cite{zhai2023determination}, the authors demonstrate uniqueness and a Lipschitz-type stability estimate for elastic and electromagnetic waves. This result holds under the assumption that the source function is piecewise constant in a homogeneous medium characterized by constant Lam\'e parameters and density. Uniqueness is established from a single boundary measurement at a fixed frequency, assuming the support of the source is a union of disjoint convex polyhedral subdomains. Additionally, there are recent results concerning increasing stability for source functions \cite{bao2020stability, yuan2023increasing}.

In this paper, we characterize the non-radiating sources located at corners within the support of underlying sources and establish the unique identifiability for the inverse source problem \eqref{problem_introduction_u_fnablafOmega} in an inhomogeneous elastic background medium. The first main result is as follows: if the support of the non-radiating source $\mathbf{f}$ contains convex corners and if $\mathbf{f}$ is $C^{\alpha}$-smooth, with $\alpha \in (0,1]$, in a neighborhood around each convex corner, then $\mathbf{f}$ vanishes at these corner points. Moreover, if $\mathbf{f}$ is $C^{1,\alpha}$-smooth for some $\alpha \in (0,1)$, the gradient $\nabla \mathbf{f}$ satisfies intricate algebraic equations that are associated with the angles of the corners.

Furthermore, when the radiating source $\mathbf{f}$ is $C^{\alpha}$-smooth with $\alpha \in (0,1]$ and the support $\Omega$ of the elastic source $\mathbf{f}$ is a convex polygon, we show that the shape and location of the support of $\mathbf{f}$ can be uniquely determined by a single far-field measurement. Simultaneously, the value of the source term $\mathbf{f}$ at the corner points of its support can be characterized. Further, when $\mathbf{f}$ is $C^{1,\alpha}$-smooth for some $\alpha \in (0,1)$, the gradient of the source function $\mathbf{f}$ can also be generically determined at the same time. Notably, all of these results are derived under the assumption that the density $\rho(\mathbf{x}) \in C^2(\mathbb{R}^2)$, which characterizes the inhomogeneous medium, is a smooth function, while the background medium is homogeneous.

The characterizations of non-radiating sources at corners presented in Theorem \ref{thm_inverse_problem_non-radiating} are local in nature. Specifically, we focus the analysis on the behavior of non-radiating sources near the corners, while a global characterization of these sources is described in terms of orthogonality relations (cf. \cite{kow2021characterization}). Furthermore, the inhomogeneity discussed in \cite{kow2021characterization} depends solely on the elasticity tensor, and the density of the inhomogeneous medium is assumed to be constant, identical to that of the homogeneous background. In contrast, this paper considers an inhomogeneous medium in which the support of the elastic source has a variable density, which differs from the background.

Compared to the previous study \cite{blaasten2018radiating}, which focuses on non-radiating elastic sources at corners under the assumption that the support of the source is in a homogeneous medium with constant density, we investigated a more general setting. In particular, we consider a variable density in the inhomogeneous medium and establish a more subtle and novel algebraic relationship between the gradient of the source and the angle defining the corner. As shown in Theorem \ref{thm_inverse_problem_non-radiating}, these algebraic characterizations of the source's gradient can be used to detect radiating sources. A radiating source $\mathbf{f}$ containing a corner radiates at all angular frequencies more easily when $\mathbf{f}$ is $C^{1,\alpha}$ continuous around the corner, as demonstrated in Corollary \ref{cor_inverse_problem_radiating}.   Regarding previous studies on uniqueness and stability analysis, such as those in \cite{blaasten2018radiating, zhai2023determination}, which assume that the background medium containing the source's support is piecewise homogeneous, our uniqueness results hold in a more general framework. Specifically, we assume that the density of the inhomogeneous medium is a general viable function, which introduces additional complexity. In this paper, we construct suitable complex geometric optics (CGO) solutions to investigate non-radiating elastic sources in inhomogeneous media with a variable density. We restrict our analysis to the microlocal behavior of the non-radiating elastic source around the corners. Since we consider both the function value and the gradient of the non-radiating elastic source at the corner points, this requires a detailed and intricate investigation of the source's behavior using CGO solutions as test functions in phase space. We utilize the characterizations of non-radiating elastic source in an inhomogeneous medium  to investigate the unique identifiability for inverse source problem \eqref{problem_introduction_u_fnablafOmega}  by a single far-field measurement.

In the following, we shall display our main results in this paper. We first introduce some geometric notations for a planar corner. Let $ \left(r, \theta\right) $ be the polar coordinates in $ \mathbb{R}^2 $. For $ \mathbf{x} = \left(x_1, x_2\right)^\top \in \mathbb{R}^2 $, the polar coordinates of $ \mathbf{x} $ are expressed as $ \left( r\cos\theta, \ r\sin\theta \right)^\top $. Given $ \mathbf{x}_0=\left(x_{01},x_{02}\right)^{\top}\in \mathbb{R}^2 $, we denote the convex sector $ W_{\mathbf{x}_0} $ and its boundaries $ \Gamma^{\pm}_{\mathbf{x}_0} $ as  
\begin{align*} 
	& W_{\mathbf{x}_0} := \left\{ \mathbf{x} \in \mathbb{R}^{2}:\mathbf{x} \neq \mathbf{x}_0, \theta_{m} < \arg \left( \left(x_1-x_{01}\right)+\text{i}\left(x_2-x_{02}\right) \right) < \theta_{M} \right\}, \notag \\
	& \Gamma^+_{\mathbf{x}_0} := \left\{ \mathbf{x} \in \mathbb{R}^{2} :\mathbf{x} \neq \mathbf{x}_0, \arg \left( \left(x_1-x_{01}\right)+\text{i}\left(x_2-x_{02}\right) \right) = \theta_{M} \right\}, \notag\\ 
	& \Gamma^-_{\mathbf{x}_0} := \left\{ \mathbf{x} \in \mathbb{R}^{2} :\mathbf{x} \neq \mathbf{x}_0, \arg \left( \left(x_1-x_{01}\right)+\text{i}\left(x_2-x_{02}\right) \right) = \theta_{m} \right\}, 
\end{align*}
where $ \theta_M - \theta_m \in\left(0,\pi\right) $ is the angle of $ W_{\mathbf{x}_0} $. Let 
\begin{align} \label{def_introduction_Dx0h}
	D_{\mathbf{x}_0,h}:=\left\{\mathbf{x} \in \mathbb{R}^2:\left| \mathbf{x}-\mathbf{x}_0 \right| < h\right\}
\end{align}
be the disk centered at $ \mathbf{x}_0 $ with radius $ h \in \mathbb{R}_+ $. We define the sector $ S_{\mathbf{x}_0,h} $ together with its boundaries $ \Gamma_{\mathbf{x}_0,h}^\pm $ and $ \Lambda_{\mathbf{x}_0,h} $ by
\begin{align} \label{def_inverse_problem_local_sector_Sh_Gammah_Lambdah}
	S_{\mathbf{x}_0,h} = W_{\mathbf{x}_0} \cap D_{\mathbf{x}_0,h}, \quad \Gamma_{\mathbf{x}_0,h}^\pm = \Gamma^\pm_{\mathbf{x}_0} \cap D_{\mathbf{x}_0,h}, \quad \Lambda_{\mathbf{x}_0,h} = W_{\mathbf{x}_0} \cap \partial D_{\mathbf{x}_0,h} .
\end{align}  
Particularly, when $ \mathbf{x}_0 $ is the origin, we abbreviate $ S_{\mathbf{x}_0,h} $, $ \Gamma_{\mathbf{x}_0,h}^\pm $, $ \Lambda_{\mathbf{x}_0,h} $, and $ D_{\mathbf{x}_0,h} $ as $ S_h $, $ \Gamma_h^\pm $, $ \Lambda_h $, and $ D_h $ respectively.

Now we are in a position to present the characterization of the non-radiating source at the corner points. It shows that the source function must vanish at these points and its gradient must satisfy intricate equations with respect to the angles of the corners. 

\begin{theorem} \label{thm_inverse_problem_non-radiating}
	Consider the elastic source scattering problem \eqref{model_global_non-nested_R2} and \eqref{Kupradze_radiation_condition} associated with \eqref{ineq_introduction_strong_convexity_condition} and $ \omega > 0 $. Assume that the positive density function $ \rho\left(\mathbf{x}\right)\in C^2\left(\mathbb{R}^2\right) $ such that $ \rho\left(\mathbf{x}\right)=1 $ for $ \mathbf{x}\in\mathbb{R}^2 \backslash \overline{D_R} $. Suppose that $ \mathbf{f}\left(\mathbf{x}\right)\in L^{\infty}\left(\mathbb{R}^2\right) $ and $ \supp\left(\mathbf{f}\right)=\overline{\Omega} $. Here, $ \Omega\subset D_R $ is a bounded Lipschitz domain with a connected complement. Let $ \mathbf{x}_0\in\partial\Omega $ such that $ \Omega\cap D_{\mathbf{x}_0, h}=S_{\mathbf{x}_0, h} $, where $ D_{\mathbf{x}_0, h} $ and $ S_{\mathbf{x}_0, h} $ are defined by \eqref{def_introduction_Dx0h} and \eqref{def_inverse_problem_local_sector_Sh_Gammah_Lambdah} with $ h\in\mathbb{R}_+ $. If $ \mathbf{f}\left(\mathbf{x}\right) \in C^{\alpha} \left(\overline{S_{\mathbf{x}_0, h}}\right) $ is a non-radiating source for some $ \alpha \in \left(0,1\right) $, i.e., $ \mathbf{u}^{\infty}\left(\hat{\mathbf{x}}\right)=\bm{0} $ for all $ \hat{\mathbf{x}} \in \mathbb{S}^1 $, then it holds that 
	\begin{align} \label{eq_thm_1_f=0}
		\mathbf{f} \left(\mathbf{x}_0\right) = \bm{0}.
	\end{align}
	If the non-radiating source $ \mathbf{f}\left(\mathbf{x}\right) \in C^{1,\alpha} \left(\overline{S_{\mathbf{x}_0, h}}\right) $ for some $ \alpha \in \left(0,1\right) $, then one has \eqref{eq_thm_1_f=0}, and there is an algebraic relationship between the gradient of $\mathbf f$ at the corner and the angle of the corner formulated as follows
	\begin{align} \label{eq_thm_1_gradientf}
		\nabla\mathbf{f}\left(\mathbf{x}_0\right)
		\begin{pmatrix}
			A \\ 
			B
		\end{pmatrix}
		+
		\begin{pmatrix}
			0 & -1 \\
			1 & 0
		\end{pmatrix}
		\nabla\mathbf{f}\left(\mathbf{x}_0\right)
		\begin{pmatrix}
			C \\ 
			D
		\end{pmatrix}
		=\mathbf{0},
	\end{align}
	where
	\begin{align*}
		\nabla\mathbf{f}\left(\mathbf{x}_0\right) = 
		\begin{pmatrix}
			\frac{\partial f_1}{\partial x_1}\left(\mathbf{x}_0\right) & \frac{\partial f_1}{\partial x_2}\left(\mathbf{x}_0\right) \\
			\frac{\partial f_2}{\partial x_1}\left(\mathbf{x}_0\right) & \frac{\partial f_2}{\partial x_2}\left(\mathbf{x}_0\right)
		\end{pmatrix},
	\end{align*}
	and 
	\begin{align} \label{def_thm_2_ABCD}
		\begin{cases}
			A=&\sin\left(\theta_M+\theta_m\right)\sin\left(\theta_M-\theta_m\right)\left(1+2\cos\left(\theta_M+\theta_m\right)\cos\left(\theta_M-\theta_m\right)\right), \\
			B=&\sin\left(\theta_M-\theta_m\right)\left(-2\cos\left(\theta_M-\theta_m\right)\cos^2\left(\theta_M+\theta_m\right)+\cos\left(\theta_M-\theta_m\right)\right. \\
			&\left.+\cos\left(\theta_M+\theta_m\right)\right), \\
			C=&\sin\left(\theta_M-\theta_m\right)\left(2\cos\left(\theta_M-\theta_m\right)\cos^2\left(\theta_M+\theta_m\right)-\cos\left(\theta_M-\theta_m\right)\right. \\
			&\left.+\cos\left(\theta_M+\theta_m\right)\right), \\
			D=&\sin\left(\theta_M+\theta_m\right)\sin\left(\theta_M-\theta_m\right)\left(-1+2\cos\left(\theta_M+\theta_m\right)\cos\left(\theta_M-\theta_m\right)\right).
		\end{cases}
	\end{align}
\end{theorem}

\begin{remark}
	The vector function $ \mathbf{f}\left(\mathbf{x}\right) $ in Theorem \ref{thm_inverse_problem_non-radiating} can be either real-valued or complex-valued. However, when $ \mathbf{f} $ is complex-valued, since $ \lambda $, $ \mu $, $ \omega $, and $ \rho\left(\mathbf{x}\right) $ are real-valued, we can treat   the real and imaginary parts of $ \mathbf{f} $ separately. This implies the conclusions in Theorem \ref{thm_inverse_problem_non-radiating} hold for a complex-valued source $\mathbf f$. For simplicity, we assume that $ \mathbf{f} $ is real-valued in the subsequent analysis. However, the conclusions obtained in this paper are applicable to a complex-valued source function. 
\end{remark}

\begin{remark}
	The linear equations  \eqref{eq_thm_1_gradientf}  can be rewritten as
	\begin{align} \label{eq_thm_1_gradientf_rewritten}
		\begin{cases}
			A \frac{\partial f_1}{\partial x_1}\left(\mathbf{x}_0\right) + B \frac{\partial f_1}{\partial x_2}\left(\mathbf{x}_0\right) - C \frac{\partial f_2}{\partial x_1}\left(\mathbf{x}_0\right) - D \frac{\partial f_2}{\partial x_2}\left(\mathbf{x}_0\right) = 0 , \\
			C \frac{\partial f_1}{\partial x_1}\left(\mathbf{x}_0\right) + D \frac{\partial f_1}{\partial x_2}\left(\mathbf{x}_0\right) + A \frac{\partial f_2}{\partial x_1}\left(\mathbf{x}_0\right) + B \frac{\partial f_2}{\partial x_2}\left(\mathbf{x}_0\right) = 0 ,
		\end{cases}
	\end{align}
	with $ A $, $ B $, $ C $, and $ D $ defined in \eqref{def_thm_2_ABCD}. Specially,  when the corner is a right corner, namely, we take  $ \theta_M=\pi/2 $ and $ \theta_m=0 $. From \eqref{eq_thm_1_gradientf_rewritten} it can be obtained that 
	$$
	\nabla \mathbf f(\mathbf x_0) = \nabla \mathbf f(\mathbf x_0)^\top,\quad \text{tr}\left(\nabla\mathbf{f}\left(\mathbf{x}_0\right)\right)=0. 
	$$
	This means that when the support of a non-radiating source $\mathbf f$ has a right corner, under certain regularity assumption on $\mathbf f$ in a neighborhood of the underlying corner, $\nabla \mathbf f(\mathbf x_0)$ is symmetric with a vanishing trace property. 
\end{remark} 

By virtue of Theorem \ref{thm_inverse_problem_non-radiating}, if the support of an elastic source $\mathbf{f}$ has a convex corner, where $\mathbf{f}$ has $C^{1,\alpha}$ continuous around the underlying corner with $\alpha \in (0,1)$, and if either \eqref{eq_thm_1_f=0} and \eqref{eq_thm_1_gradientf_rewritten} is   violated, then $\mathbf f$ must radiate at every  frequency.

\begin{corollary} \label{cor_inverse_problem_radiating}
	Under the same settings as Theorem \ref{thm_inverse_problem_non-radiating}, \textcolor{black}{suppose that there exists a point $ \mathbf{x}_0\in\partial\Omega $ such that $ \Omega\cap D_{\mathbf{x}_0, h}=S_{\mathbf{x}_0, h} $, where $ D_{\mathbf{x}_0, h} $ and $ S_{\mathbf{x}_0, h} $ are defined by \eqref{def_introduction_Dx0h} and \eqref{def_inverse_problem_local_sector_Sh_Gammah_Lambdah} with $ h\in\mathbb{R}_+ $. Furthermore,} if the source term $ \mathbf{f}\left(\mathbf{x}\right) \in C^{1,\alpha} \left(\overline{S_{\mathbf{x}_0,h}}\right) $ with 	some $ \alpha \in \left(0,1\right) $ satisfies at least one of the following conditions:
	\begin{align} 
		\label{eq:1.13}
		\mathbf{f} \left(\mathbf{x}_0\right) &\neq \bm{0}, \\
		\label{eq:1.14}
		A \frac{\partial f_1}{\partial x_1}\left(\mathbf{x}_0\right) + B \frac{\partial f_1}{\partial x_2}\left(\mathbf{x}_0\right) - C& \frac{\partial f_2}{\partial x_1}\left(\mathbf{x}_0\right) - D \frac{\partial f_2}{\partial x_2}\left(\mathbf{x}_0\right) \neq 0 , \\
		\label{eq:1.15}
		C \frac{\partial f_1}{\partial x_1}\left(\mathbf{x}_0\right) + D \frac{\partial f_1}{\partial x_2}\left(\mathbf{x}_0\right) + A& \frac{\partial f_2}{\partial x_1}\left(\mathbf{x}_0\right) + B \frac{\partial f_2}{\partial x_2}\left(\mathbf{x}_0\right) \neq 0 ,
	\end{align}	
	where $ A $, $ B $, $ C $, and $ D $ are defined as \eqref{def_thm_2_ABCD}, then $ \mathbf{f}\left(\mathbf{x}\right) $ is radiating.
\end{corollary}

\begin{remark}
	\textcolor{black}{It is noted that in our study we assume that $\mathbf f \in L^\infty(\mathbb R^2)$ with ${\rm supp}(\mathbf f)=\overline{\Omega}$. Therefore if $\mathbf f$ is not continuous across $\partial\Omega$, it is readily to claim that $\mathbf f$ does not vanish at any boundary point of $\partial \Omega$, which means that \eqref{eq:1.13} is satisfied when $\Omega$ contain a convex sectorial point $\mathbf x_0$.} However, it is emphasized that it is possible to have that $\mathbf{f}\left(\mathbf{x}\right) $ vanishes at some boundary points of its support. Hence \eqref{eq:1.13} cannot be fulfilled under certain physical scenarios. 	\textcolor{black}{When $\Omega$ contain a convex sectorial point $\mathbf x_0$, $\nabla \mathbf f $ is H\"older continuous in $\overline{S_{\mathbf x_0,h}}$ and is not continuous across $\partial \Omega$, we can claim that} the assumptions \eqref{eq:1.14} and \eqref{eq:1.15} are typically valid from a physical perspective, as they depend on specific values of the gradient of the source and the angles $\theta_m$ and $\theta_M$ of the convex sector $S_{\mathbf{x}_0,h}$. \textcolor{black}{We observe that the coefficients $A$, $B$, $C$ and $D$ of \eqref{eq:1.14} and \eqref{eq:1.15} are related to the sinusoidal and cosinoidal functions of the angles $\theta_m$ and $\theta_M$ of the convex sector $  S_{\mathbf{x}_0,h}  $, which are  highly intricate coefficients. If the assumptions \eqref{eq:1.14} and \eqref{eq:1.15} are violated, which implies that \eqref{eq_thm_1_gradientf_rewritten} is satisfied. Hence the gradient of the source term at the corner point must fulfill a linear system \eqref{eq_thm_1_gradientf_rewritten} with intricate coefficients  $A$, $B$, $C$ and $D$, which is generally difficult to satisfy when $ \nabla \mathbf{f}\left(\mathbf{x}_0\right) \neq \mathbf{0}$. Consequently, \eqref{eq:1.14} and \eqref{eq:1.15} typically hold  which implies that $ \mathbf{f} $ is radiating at any angular frequency, provided that $ \nabla \mathbf{f} $ does not vanish at $\mathbf{x}_0$ and is H\"older continuous in the sector $  \overline{S_{\mathbf{x}_0,h}}  $. }
\end{remark}

An important application of Theorem \ref{thm_inverse_problem_non-radiating} is to investigate the geometrical inverse problem defined by \eqref{problem_introduction_u_fnablafOmega} by a single far-field measurement.  Before establishing the local and global  unique recovery results for the inverse source problem \eqref{problem_introduction_u_fnablafOmega} in Theorems \ref{thm_inverse_problem_local_uniqueness} and \ref{thm_inverse_problem_global_uniqueness}, we introduce the admissible class of elastic sources for our subsequent study. 

\begin{definition} \label{def_admissible_elastic_source_local}
	The pair $ \left(\Omega; \mathbf{f}\right) $ is referred to as \textcolor{black}{an admissible elastic source} if it satisfies the following conditions:
	\begin{enumerate} [(1)]
		\item $ \Omega\subset D_R $ with $R\in \mathbb{R}_+$ is a bounded Lipschitz domain in $\mathbb{R}^2 $ with a connected complement. The vector function $ \mathbf{f}\left(\mathbf{x}\right)\in L^{\infty}\left(\mathbb{R}^2\right) $ and $ \supp\left(\mathbf{f}\right)=\overline{\Omega} $. 
		\item If there is a point $ \mathbf{x}_0=\left(x_{01},x_{02}\right)^{\top}\in\partial\Omega $ satisfying $ \Omega\cap D_{\mathbf{x}_0,h}=S_{\mathbf{x}_0,h} $, where the sector $ S_{\mathbf{x}_0,h} $ is given by \eqref{def_inverse_problem_local_sector_Sh_Gammah_Lambdah} with a sufficiently small $ h\in\mathbb{R}_+ $, then either of the following two assumptions is satisfied
		\begin{itemize}
			\item[(2.1)] $ \mathbf{f}\left(\mathbf{x}\right)\in C^{\alpha}\left(\overline{S_{\mathbf{x}_0,h}}\right) $ for some $ 0<\alpha<1 $ and fulfills  \eqref{eq:1.13}   in Corollary \ref{cor_inverse_problem_radiating} at $ \mathbf{x}_0 $;
			\item[(2.2)] $ \mathbf{f}\left(\mathbf{x}\right)\in C^{1, \alpha}\left(\overline{S_{\mathbf{x}_0,h}}\right) $ for some $ 0<\alpha<1 $ and fulfills at least one of the conditions \eqref{eq:1.13}  to \eqref{eq:1.15}  in Corollary \ref{cor_inverse_problem_radiating} at $ \mathbf{x}_0 $.
		\end{itemize}
	\end{enumerate}
\end{definition}

The following theorem demonstrates that the local uniqueness result for determining the shape of the source from a single far-field pattern.

\begin{theorem} \label{thm_inverse_problem_local_uniqueness}
Consider the elastic source scattering problem governed by \eqref{model_global_non-nested_R2} and \eqref{Kupradze_radiation_condition}. Let $\left(\Omega; \mathbf{f}\right)$ and $\left(\Omega^{\prime}; \mathbf{f}^{\prime}\right)$ be admissible elastic sources described by Definition \ref{def_admissible_elastic_source_local}. Let $ \mathbf{u}_{\infty} $ and $ \mathbf{u}_{\infty}^{\prime} $ denote the elastic far-field patterns corresponding to sources $\left(\Omega; \mathbf{f}\right)$ and $\left(\Omega^{\prime}; \mathbf{f}^{\prime}\right)$ respectively. If the far-field patterns coincide, i.e.,
	\begin{align*} 
		\mathbf{u}_{\infty}\left(\hat{\mathbf{x}}\right) =\mathbf{u}_{\infty}^{\prime} \left(\hat{\mathbf{x}}\right) , \quad  \mathbf{x}\in\mathbb{S}^1, 
	\end{align*}
\textcolor{black}{then the symmetric difference $ \Omega\Delta\Omega^{\prime}:=\left(\Omega\backslash\Omega^{\prime}\right)\cup\left(\Omega^{\prime}\backslash\Omega\right) $ cannot contain a convex sectorial corner point $\mathbf{x}_0$, where it has the properties that
	\begin{enumerate} [(1)]
		\item The convex sector $ S_{\mathbf{x}_0,h} $ of $\Omega\Delta\Omega^{\prime}$, for a sufficiently small $h>0$, has both of its boundaries $ \Gamma_{\mathbf{x}_0,h}^\pm $ lying simultaneously on $ \partial\Omega$ or both on $ \partial\Omega^{\prime} $.
		\item For every point $\mathbf{x}$ on the boundaries $ \Gamma_{\mathbf{x}_0,h}^\pm $, there exists an unbounded path $\gamma\subset\mathbb{R}^2\backslash\left(\overline{\Omega}\cup\overline{\Omega^\prime}\right)$ connecting $\mathbf{x}$ to infinity.
\end{enumerate}}
\end{theorem}

We now present the global uniqueness results for the inverse source problems. Under a single far-field measurement, we establish uniqueness in determining both the shape and location of the support of the radiating source, assuming the support is a convex polygon. Additionally, we demonstrate that the value of the source function at each corner of its support can be simultaneously identified using the same single far-field measurement. Furthermore, when the source function exhibits $C^{1,\alpha}$ regularity in the neighborhood of a corner, the gradient of the source function at that corner can also be generically determined.

\begin{theorem} \label{thm_inverse_problem_global_uniqueness}
	Consider the elastic scattering  problem \eqref{model_global_non-nested_R2} and \eqref{Kupradze_radiation_condition}. For the convex polygons $ \Omega $ and $ \Omega^{\prime} $, assume that $\left(\Omega; \mathbf{f}\right)$ and $\left(\Omega^{\prime}; \mathbf{f}^{\prime}\right)$ are two admissible elastic sources described by Definition \ref{def_admissible_elastic_source_local}.  Let $ \mathbf{u}_{\infty} $ and $ \mathbf{u}_{\infty}^{\prime} $ denote the far-field patterns corresponding to sources $\left(\Omega; \mathbf{f}\right)$ and $\left(\Omega^{\prime}; \mathbf{f}^{\prime}\right)$. If the far-field patterns are identical, i.e.,
	\begin{align*} 
		\mathbf{u}_{\infty}\left(\hat{\mathbf{x}}\right) =\mathbf{u}_{\infty}^{\prime} \left(\hat{\mathbf{x}}\right), \quad \mathbf{x}\in\mathbb{S}^1,
	\end{align*} 
	then $ \Omega = \Omega^{\prime} $. Furthermore,  for every vertex $ \mathbf{x}_{i} $ of $ \Omega $, it follows that 
	\begin{align} \label{eq_thm_3_f=fprime}
		\mathbf{f}\left(\mathbf{x}_{i}\right)=\mathbf{f}^{\prime}\left(\mathbf{x}_{i}\right).
	\end{align}
	When $ \mathbf{f}\left(\mathbf{x}\right)\textcolor{black}{,\mathbf{f}^{\prime}\left(\mathbf{x}\right)}\in C^{1, \alpha}\left(\overline{S_{\mathbf{x}_{i},h}}\right) $ for some $ 0<\alpha<1 $ and $ h\in\mathbb{R}_+ $, it also satisfies
	\begin{align} \label{eq_thm_3_gradientf=gradientfprime}
		\begin{cases}
			A \frac{\partial \left(f_1-f^{\prime}_1\right)}{\partial x_1}\left(\mathbf{x}_{i}\right) + B \frac{\partial \left(f_1-f^{\prime}_1\right)}{\partial x_2}\left(\mathbf{x}_{i}\right) - C \frac{\partial \left(f_2-f^{\prime}_2\right)}{\partial x_1}\left(\mathbf{x}_{i}\right) - D \frac{\partial \left(f_2-f^{\prime}_2\right)}{\partial x_2}\left(\mathbf{x}_{i}\right) = 0 , \\
			C \frac{\partial \left(f_1-f^{\prime}_1\right)}{\partial x_1}\left(\mathbf{x}_{i}\right) + D \frac{\partial \left(f_1-f^{\prime}_1\right)}{\partial x_2}\left(\mathbf{x}_{i}\right) + A \frac{\partial \left(f_2-f^{\prime}_2\right)}{\partial x_1}\left(\mathbf{x}_{i}\right) + B \frac{\partial \left(f_2-f^{\prime}_2\right)}{\partial x_2}\left(\mathbf{x}_{i}\right) = 0 ,
		\end{cases}
	\end{align}
	where the coefficients $ A $, $ B $, $ C $, and $ D $ are defined in \eqref{def_thm_2_ABCD} correspondingly.	
\end{theorem}

\begin{remark}
	Equation \eqref{eq_thm_3_gradientf=gradientfprime} presents an intricate result, stating that \textcolor{black}{if $\nabla \mathbf{f}\left(\mathbf{x}\right)$ and $\nabla \mathbf{f}^{\prime}\left(\mathbf{x}\right)$ are H\"older continuous in $\overline{S_{\mathbf x_{i},h}}$, then} the gradient  $ \nabla \left(\mathbf{f}-\mathbf{f}^{\prime}\right) $ must satisfy algebraic equations, \textcolor{black}{wherein the coefficients $A$, $B$, $C$ and $D$ of  \eqref{eq_thm_3_gradientf=gradientfprime}  are associated with the sinusoidal and cosinoidal functions of the angles $\theta_{i,m}$ and $\theta_{i,M}$ of each convex sector $  S_{\mathbf{x}_{i},h}  $, which are  highly intricate.} From a physical intuition perspective, the situation that the components of $ \nabla \left(\mathbf{f}-\mathbf{f}^{\prime}\right) $ at the corner \textcolor{black}{satisfy \eqref{eq_thm_3_gradientf=gradientfprime} yet not} equal to $\mathbf{0}$ is generally unlikely to occur, \textcolor{black}{as the coefficients of \eqref{eq_thm_3_gradientf=gradientfprime}, defined  by \eqref{def_thm_2_ABCD}, are highly intricate coefficients.} Consequently, \textcolor{black}{\eqref{eq_thm_3_gradientf=gradientfprime} implies that} $ \grad\mathbf{f} $ generically  equals $ \grad\mathbf{f}^{\prime} $ at every convex corner. Furthermore, from \eqref{eq_thm_3_gradientf=gradientfprime}, we observe that for every vertex $ \mathbf{x}_{i} $ of $ \partial \Omega $, if any two components of  $ \grad\mathbf{f}$ are identical to their counterparts  in  $ \grad\mathbf{f}^{\prime}  $, then the remaining components must also be identical. For example, if $ \frac{\partial f_1}{\partial x_1}(\mathbf{x}_{i}) = \frac{\partial f^{\prime}_1}{\partial x_1}(\mathbf{x}_{i}) $ and $ \frac{\partial f_2}{\partial x_2}(\mathbf{x}_{i}) = \frac{\partial f^{\prime}_2}{\partial x_2}(\mathbf{x}_{i}) $, then it follows that $ \frac{\partial f_1}{\partial x_2}(\mathbf{x}_{i}) = \frac{\partial f^{\prime}_1}{\partial x_2}(\mathbf{x}_{i}) $ and $ \frac{\partial f_2}{\partial x_1}(\mathbf{x}_{i}) = \frac{\partial f^{\prime}_2}{\partial x_1}(\mathbf{x}_{i}) $.
\end{remark}

The rest of this paper is organized as follows. In Section \ref{section_inverse_problem}, we show the proofs of theorems stated above. In Appendix \ref{section_appendix}, we present the construction and estimates of the CGO solution involved in Section \ref{section_inverse_problem}. 

\section{Proof of theorems} \label{section_inverse_problem}

In this section, we shall prove the theorems stated in Section \ref{section_introduction}. To begin with, we present the CGO solution $ \mathbf{u}_0 \left(\mathbf{x}\right) $ as follows, which works as a test function in our analysis. 

\begin{lemma} \label{lem_CGOs}
	Suppose that the positive function $ \rho\left(\mathbf{x}\right)\in C^2\left(\mathbb{R}^2\right) $ satisfies $ \rho\left(\mathbf{x}\right)=1 $ for $ \mathbf{x}\in\mathbb{R}^2 \backslash \overline{D_R} $. Let $ 0 < R < R_1 < R^{\prime} $, and $ D_{R_1}:=\left\{\mathbf{x} \in \mathbb{R}^2:\left| \mathbf{x} \right|< R_1\right\} $. Recall that $ k_{\mathrm{s}} $ is defined in \eqref{eq_introduction_kpks}. \textcolor{black}{Let $ \mathbf{d} $ and $ \mathbf{d}^{\perp}\in \mathbb{S}^1 $ be} the unit vectors fulfilling that $ \mathbf{d}^{\perp} $ is perpendicular to $ \mathbf{d} $. Assume that 
	\begin{align} \label{def_IP_zeta_eta}
		\bm{\zeta} := \tau \mathbf{d} + \mathrm{i}\sqrt{ \tau^2 + k_{\mathrm{s}}^2 } \mathbf{d}^{\perp}, \quad \bm{\eta} := - \mathrm{i} \sqrt{ 1 + \frac{k_{\mathrm{s}}^2}{\tau^2}} \mathbf{d} + \mathbf{d}^\perp,
	\end{align}
	where $ \tau $ is a constant. If $\tau$ is large enough, then there exists a function  $ \mathbf{R}\left(\mathbf{x}\right)\in C^2\left(D_{R_1}\right) $ such that the Complex Geometrical Optics solution given by
	\begin{align} \label{def_CGOs_u0}
		\mathbf{u}_0 \left(\mathbf{x}\right) := e^{\bm{\zeta} \cdot \mathbf{x}} \bm{\eta} + e^{\bm{\zeta} \cdot \mathbf{x}} \mathbf{R}\left(\mathbf{x}\right), 
	\end{align}
	satisfies the Navier equation 
	\begin{align} \label{eq_CGOs_Navier_equation_no_source}
		\mathcal{L} \mathbf{u} + \omega^2 \rho\left(\mathbf{x}\right) \mathbf{u} = \bm{0} \quad \text{ in } D_{R_1}.
	\end{align}
	Furthermore, when $ \tau $ is sufficiently large, the following estimations hold
	\begin{align} \label{est_inverse_problem_CGO}
		\left\| \mathbf{R} \left(\mathbf{x}\right) \right\|_{L^2 \left(D_{R_1}\right)} \leqslant \frac{c}{\tau}, \quad \left\| \nabla \mathbf{R} \left(\mathbf{x}\right) \right\|_{L^2 \left(D_{R_1}\right)} \leqslant c, \quad \left\| \nabla^2 \mathbf{R} \left(\mathbf{x}\right) \right\|_{L^2 \left(D_{R_1}\right)} \leqslant c \tau, 
	\end{align}
	where $ c $ is a constant depending only on $ R_1 $, $ R^{\prime} $, $ \lambda $, $ \mu $, $ \omega $, and $ \rho\left(\mathbf{x}\right) $.
\end{lemma}

Since the proof of Lemma \ref{lem_CGOs} is technical, we put it in Appendix \ref{section_appendix}.

\textcolor{black}{\begin{remark}
		In Lemma \ref{lem_CGOs}, we construct a CGO solution \(\mathbf{u}_0(\mathbf{x})\) to the Navier equation \eqref{eq_CGOs_Navier_equation_no_source} in $\mathbb R^2$, providing \( L^2 \)-norm estimates for \(\mathbf{R}(\mathbf{x})\), \(\nabla \mathbf{R}(\mathbf{x})\), and \(\nabla^2 \mathbf{R}(\mathbf{x})\). Compared to prior results, such as those in \cite{barcelo2018uniqueness, hahner2002uniqueness}, which focus on CGO solutions in $\mathbb R^3$ under different regularity assumptions on \(\rho(\mathbf{x})\), our work offers a refined characterization of CGO solution to the two-dimensional setting, notably including an $L^2$-norm estimate for $\nabla^2\mathbf{R}\left(\mathbf{x}\right)$.
\end{remark}}

In order to characterize the behavior of non-radiating elastic sources at the corner points, we first consider the following PDE system
\begin{align} \label{model_part_v_Sh}
	\begin{cases}
		\mathcal{L} \mathbf{v} + \omega^2 \rho\left(\mathbf{x}\right) \mathbf{v} = \mathbf{f} & \text{in } S_{\mathbf{x}_0,h} , \\
		\mathbf{v} = \bm{0}, \ T_{\bm{\nu}} \mathbf{v} = \bm{0} & \text {on } \Gamma_{\mathbf{x}_0,h}^\pm,
	\end{cases}
\end{align}
Here, $ \omega>0 $, the positive function $ \rho\left(\mathbf{x}\right)\in C^2\left(\mathbb{R}^2\right) $ and $ \rho\left(\mathbf{x}\right)=1 $ for $ \mathbf{x}\in\mathbb{R}^2 \backslash D_R $. The sector $ S_{\mathbf{x}_0,h} $ and its boundaries $ \Gamma_{\mathbf{x}_0,h}^\pm $ are defined in \eqref{def_inverse_problem_local_sector_Sh_Gammah_Lambdah} and $ S_{\mathbf{x}_0,h}\subset D_R $.  $ T_{\bm{\nu}} $ with the conormal derivative $ T_{\bm{\nu}}\mathbf{u} := \lambda\left(\nabla\cdot\mathbf{u}\right)\bm{\nu} + 2\mu\left(\nabla^{\text{s}}\mathbf{u}\right)\bm{\nu} $, where $ \bm{\nu} $ is the outward unit normal to the boundary $ \partial \Gamma^\pm_{\mathbf x_0,h} $ and  $ \nabla^{\text{s}}\mathbf{u} := \left(\nabla\mathbf{u}+\nabla\mathbf{u}^{\top}\right)/2 $ denotes the strain tensor.

Since the operator $ \mathcal{L} $ is invariant under rigid motion, without loss of generality, we assume that $ \mathbf{x}_0 = \bm{0} $ in the subsequent analysis. By Lemma \ref{lem_CGOs} and the first Betti's formula, we obtain that
\begin{align} \label{eq_fu0=u0Tnuv_original}
	\int_{S_h} \mathbf{f} \cdot \mathbf{u}_0 \mathrm{d}\mathbf{x} = \int_{S_h} \mathbf{u}_0 \cdot \mathcal{L} \mathbf{v} - \mathbf{v} \cdot \mathcal{L} \mathbf{u}_0 \mathrm{d}\mathbf{x} = \int_{\Lambda_h} \mathbf{u}_0 \cdot T_{\bm{\nu}} \mathbf{v} - \mathbf{v} \cdot T_{\bm{\nu}} \mathbf{u}_0 \mathrm{d}\sigma.
\end{align}
Suppose that $ \mathbf{f}\left(\mathbf{x}\right) \in C^{1,\alpha} \left(\overline{S_h}\right) $. We have
\begin{align} \label{def_IP_f_C1alpha}
	\mathbf{f} \left(\mathbf{x}\right) = \mathbf{f} \left(\bm{0}\right) + \nabla \mathbf{f} \left(\bm{0}\right) \mathbf{x} + \bm{\delta} \mathbf{f} \left(\mathbf{x}\right) , \quad \mathbf{x}\in\overline{S_h}, 
\end{align}
where $ \mathbf{f}\left(\bm{0}\right) $ and $\bm{\delta}\mathbf{f}\left(\mathbf{x}\right) \in\mathbb{R}^2 $, $ \nabla\mathbf{f} \left(\bm{0}\right) \in\mathbb{R}^{2\times2} $, and $ \left| \bm{\delta} \mathbf{f} \left(\mathbf{x}\right) \right| \leqslant \left\| \mathbf{f} \left(\mathbf{x}\right) \right\|_{C^{1,\alpha}\left(\overline{S_h}\right)} \left| \mathbf{x} \right|^{1+\alpha} $. Then it follows that
\begin{align} \label{eq_IP_fexpansion_Bettiformula}
	\mathbf{f} \left(\bm{0}\right) \cdot \int_{S_h} \mathbf{u}_0 \mathrm{d}\mathbf{x} = & \int_{\Lambda_h} \mathbf{u}_0 \cdot T_{\bm{\nu}} \mathbf{v} - \mathbf{v} \cdot T_{\bm{\nu}} \mathbf{u}_0 \mathrm{d}\sigma - \int_{S_h} \left( \nabla \mathbf{f} \left(\bm{0}\right) \mathbf{x} \right) \cdot \mathbf{u}_0 \mathrm{d}\mathbf{x} \notag \\
	& - \int_{S_h} \bm{\delta} \mathbf{f} \left(\mathbf{x}\right) \cdot \mathbf{u}_0 \mathrm{d}\mathbf{x}.
\end{align}
In the following, we shall estimate each term of \eqref{eq_IP_fexpansion_Bettiformula} with respect to the parameter $\tau$ of $\mathbf u_0$ as $\tau \rightarrow \infty$. 

With the help of the Laplace transformation and the negative order analysis of exponential functions, we have the following proposition \textcolor{black}{(cf. \cite{diao2022further})}.

\begin{proposition} \label{prop_inverse_problem_integral_re}
	For positive constants $ \alpha $ and $ \varepsilon $, as long as $ \mathrm{Re}  \left(\mu\right) \geqslant \frac{2\alpha}{e} $, we have
	\begin{align*} 
		\int_{0}^{\varepsilon} r^\alpha e^{- \mu r} \mathrm{d}r = \frac{\Gamma \left( \alpha + 1 \right)}{\mu^{\alpha + 1}} - I_R,
	\end{align*}
	where $ \Gamma \left(s\right)$ denotes the Gamma function, $ I_R = \int_{\varepsilon}^{+\infty} r^\alpha e^{-\mu r} \mathrm{d}r $, $ \left| I_R \right| \leqslant \frac{2}{\mathrm{Re}\left(\mu\right)} e^{ - \frac{\varepsilon}{2} \mathrm{Re} \left(\mu\right)} $. 
\end{proposition}

Recall the definition of the geometric notations $ S_h $, $ \Gamma^{\pm}_h $ and $ \Lambda_h $ in \eqref{def_inverse_problem_local_sector_Sh_Gammah_Lambdah}. We select the unit vector $ \mathbf{d} $ appearing in Lemma \ref{lem_CGOs} such that 
\begin{align} \label{ineq_IP_dcodtx}
	-1 \leqslant \mathbf{d} \cdot \hat{\mathbf{x}} \leqslant - \delta ,\quad \delta \in \left(0,1\right),
\end{align}
for each $ \hat{\mathbf{x}} \in \Lambda_1 $. Then the estimate about $ \mathbf{u}_0\left(\mathbf{x}\right) $ on $ \Lambda_h $ is presented as follows.

\begin{lemma} \label{lem_IP_est_int_Lamdah_u0Tv-vTu0} 
	Let $ S_h $ and $ \Lambda_h $ be defined in \eqref{def_inverse_problem_local_sector_Sh_Gammah_Lambdah}. Suppose that $ S_h\subset D_R $. Let $ \mathbf{u}_0\left(\mathbf{x}\right) $ be defined in \eqref{def_CGOs_u0} with the unit vector $ \mathbf{d} $ satisfying \eqref{ineq_IP_dcodtx}. Then for $ \mathbf{v}\left(\mathbf{x}\right) \in H^1\left(\mathbb{R}^2\right) $, as $ \tau $ is large enough, it holds that
	\begin{align} \label{est_IP_int_Lamdah_u0Tv-vTu0} 
		\left| \int_{\Lambda_h} \left(\mathbf{u}_0 \cdot \left( T_{\bm{\nu}} \mathbf{v} \right) - \mathbf{v} \cdot \left(T_{\bm{\nu}} \mathbf{u}_0\right)\right) \mathrm{d}\sigma \right| \leqslant c\tau e^{- \delta h \tau},
	\end{align}  
	where the constant $ c $ is independent of $ \tau $.
\end{lemma}

\begin{proof}
	By virtue of the dual norm, the definition of $ T_{\bm{\nu}} $, and the trace theorem, one can derive that
	\begin{align*} 
		& \left| \int_{\Lambda_h} \mathbf{u}_0 \cdot \left( T_{\bm{\nu}} \mathbf{v} \right) \mathrm{d}\sigma \right| \leqslant \left\| \mathbf{u}_0 \right\|_{H^{\frac{1}{2}} \left(\Lambda_h\right)} \left\| T_{\bm{\nu}} \mathbf{v} \right\|_{H^{-\frac{1}{2}} \left(\Lambda_h\right)} \leqslant c \left\| \mathbf{u}_0 \right\|_{H^1 \left(\Lambda_h\right)} \left\| \mathbf{v} \right\|_{H^1 \left(S_h\right)}, \\
		& \left| \int_{\Lambda_h} \mathbf{v} \cdot \left( T_{\bm{\nu}} \mathbf{u}_0 \right) \mathrm{d}\sigma \right| \leqslant \left\| \mathbf{v} \right\|_{L^2 \left(\Lambda_h\right)} \left\| T_{\bm{\nu}} \mathbf{u}_0 \right\|_{L^2 \left(\Lambda_h\right)} \leqslant c \left\| \mathbf{v} \right\|_{H^{1} \left(S_h\right)} \left\| \nabla\mathbf{u}_0 \right\|_{L^2 \left(\Lambda_h\right)}.
	\end{align*}
	where $ c $ is a constant independent of $ \tau $. Divide $ \mathbf{u}_0\left(\mathbf{x}\right) $ into 
	\textcolor{black}{\begin{align} \label{def_IP_u01u02}
			\mathbf{u}_{01}\left(\mathbf{x}\right) := e^{\bm{\zeta} \cdot \mathbf{x}} \bm{\eta}, \quad \mathbf{u}_{02}\left(\mathbf{x}\right) := e^{\bm{\zeta} \cdot \mathbf{x}} \mathbf{R}\left(\mathbf{x}\right).
	\end{align}}
	Notice that 
	\begin{align*} 
		\left| e^{\bm{\zeta} \cdot \mathbf{x}} \right| = \left| e^{ \tau h \mathbf{d}\cdot\hat{\mathbf{x}} } e^{\mathrm{i}\sqrt{ \tau^2 + k_{\mathrm{s}}^2 } \mathbf{d}^{\perp}\cdot\mathbf{x}} \right| \leqslant e^{-\delta h \tau} , \quad \mathbf{x}\in\Lambda_h.
	\end{align*}
	As $ \tau $ is large enough, by direct calculation, we have
	\begin{align*} 
		\left\| \mathbf{u}_{01} \right\|_{L^2 \left(\Lambda_h\right)}  
		\leqslant c e^{-\delta h \tau}, \quad \left\| \nabla\mathbf{u}_{01} \right\|_{L^2 \left(\Lambda_h\right)} \leqslant c \tau e^{-\delta h \tau}.
	\end{align*}
	The trace theorem and Lemma \ref{lem_CGOs} yield that
	\begin{align*} 
		\left\| \mathbf{u}_{02} \right\|_{L^2 \left(\Lambda_h\right)} & \leqslant c e^{-\delta h \tau} \left\| \mathbf{R}\left(\mathbf{x}\right) \right\|_{H^1 \left(S_h\right)} \leqslant c e^{-\delta h \tau},\\ 
		\left\| \nabla\mathbf{u}_{02} \right\|_{L^2 \left(\Lambda_h\right)} & \leqslant \left\| e^{\bm{\zeta} \cdot \mathbf{x}} \mathbf{R}\left(\mathbf{x}\right) \bm{\zeta}^{\top} \right\|_{L^2 \left(\Lambda_h\right)} + \left\| e^{\bm{\zeta} \cdot \mathbf{x}} \nabla\mathbf{R}\left(\mathbf{x}\right) \right\|_{L^2 \left(\Lambda_h\right)} \notag \\
		& \leqslant c \tau e^{-\delta h \tau} \left\| \mathbf{R}\left(\mathbf{x}\right) \right\|_{H^1 \left(S_h\right)}	+ ce^{-\delta h \tau} \left\| \mathbf{R}\left(\mathbf{x}\right) \right\|_{H^2 \left(S_h\right)} \leqslant c \tau e^{-\delta h \tau}.
	\end{align*}
	Then it turns out that
	\begin{align*} 
		\left\| \mathbf{u}_0 \right\|_{H^1 \left(\Lambda_h\right)} \leqslant c \tau e^{-\delta h \tau}.
	\end{align*}
	Since $ \mathbf{v} \in H^1\left(\mathbb{R}^2\right) $, we arrive at the estimate \eqref{est_IP_int_Lamdah_u0Tv-vTu0}, which completes the proof.
\end{proof}

In the following, we give estimates about the real-valued vector function $ \mathbf{f}\left(\mathbf{x}\right)\in C^{1,\alpha}\left(\overline{S_h}\right) $ for some $ \alpha\in\left(0,1\right) $, which can be expanded as \eqref{def_IP_f_C1alpha}.

\begin{lemma} \label{lem_IP_est_int_Sh_xu0_int_Sh_delatfu0}
	Let the sector $ S_h $ be defined as in \eqref{def_inverse_problem_local_sector_Sh_Gammah_Lambdah} and $ S_h\subset D_R $. Let $ \mathbf{f}\left(\mathbf{x}\right)\in C^{1,\alpha}\left(\overline{S_h}\right) $ for some $ \alpha\in\left(0,1\right) $ with its expansion \eqref{def_IP_f_C1alpha}. Suppose that $ \tau $ is sufficiently large. Then for $ \mathbf{u}_0\left(\mathbf{x}\right) $ defined in \eqref{def_CGOs_u0} with the unit vector $ \mathbf{d} $ fulfilling \eqref{ineq_IP_dcodtx}, we have the estimates
	\begin{align} \label{est_IP_int_Sh_gradientf0xu0}
		\left| \int_{S_h} \left(\nabla \mathbf{f} \left(\bm{0}\right) \mathbf{x}\right) \cdot \mathbf{u}_0\left(\mathbf{x}\right) \mathrm{d}\mathbf{x} \right| \leqslant \frac{c}{\tau^3} ,\quad
		\left| \int_{S_h} \bm{\delta} \mathbf{f} \left(\mathbf{x}\right) \cdot \mathbf{u}_0\left(\mathbf{x}\right) \mathrm{d}\mathbf{x} \right| \leqslant \frac{c}{\tau^{3+\alpha}} ,
	\end{align}
	where $ c $ is a constant independent of $ \tau $.
\end{lemma}

\begin{proof}
	By the aid of \eqref{def_IP_f_C1alpha} and \eqref{def_CGOs_u0}, one can derive 
	\begin{align*} 
		\left| \int_{S_h} \left(\nabla \mathbf{f} \left(\bm{0}\right) \mathbf{x}\right) \cdot \mathbf{u}_0\left(\mathbf{x}\right) \mathrm{d}\mathbf{x} \right| \leqslant c \int_{S_h} \left|\mathbf{x}\right| \left|e^{\bm{\zeta} \cdot \mathbf{x}} \bm{\eta} \right| \mathrm{d}\mathbf{x} + c \int_{S_h} \left|\mathbf{x}\right| \left|e^{\bm{\zeta} \cdot \mathbf{x}} \mathbf{R}\left(\mathbf{x}\right)\right| \mathrm{d}\mathbf{x},
	\end{align*}
	where the constant $ c>0 $ is independent of $ \tau $. With the help of 
	\begin{align} \label{ineq_IP_ezetax_Sh}
		\left| e^{\bm{\zeta} \cdot \mathbf{x}} \right| = \left| e^{ \tau \left|\mathbf{x}\right| \mathbf{d}\cdot\hat{\mathbf{x}} } e^{\mathrm{i}\sqrt{ \tau^2 + k_{\mathrm{s}}^2 } \mathbf{d}^{\perp}\cdot\mathbf{x}} \right| \leqslant e^{-\delta \left|\mathbf{x}\right| \tau} , \quad \mathbf{x}\in S_h,
	\end{align}
	the polar coordinate transformation, the boundedness of $ \bm{\eta} $, and Proposition \ref{prop_inverse_problem_integral_re}, we obtain 
	\begin{align*} 
		\int_{S_h} \left|\mathbf{x}\right| \left|e^{\bm{\zeta} \cdot \mathbf{x}} \bm{\eta} \right| \mathrm{d}\mathbf{x} \leqslant c \int_{\theta_m}^{\theta_M} \int_{0}^{h} r^2 e^{-\delta \tau r} \mathrm{d}r \mathrm{d}\theta 
		\leqslant \frac{c}{\tau^3} ,
	\end{align*}
	where $ \tau $ is large enough and $ c $ is a constant independent of $ \tau $. Similarly, from \eqref{est_inverse_problem_CGO}, it turns out that
	\begin{align*} 
		\int_{S_h} \left|\mathbf{x}\right| \left|e^{\bm{\zeta} \cdot \mathbf{x}} \mathbf{R}\left(\mathbf{x}\right)\right| \mathrm{d}\mathbf{x} \leqslant \left\| \mathbf{x} e^{\bm{\zeta} \cdot \mathbf{x}} \right\|_{L^2 \left(S_h\right)} \left\| \mathbf{R}\left(\mathbf{x}\right) \right\|_{L^2 \left(S_h\right)} \leqslant \frac{c}{\tau^3} .
	\end{align*}
	The estimate about $ \bm{\delta} \mathbf{f} \left(\mathbf{x}\right) $ in \eqref{est_IP_int_Sh_gradientf0xu0} can be derived by the analogous strategy, so we omit it here.
	
	The proof is complete.
\end{proof}

Next, we give another estimate concerning the dot product of $ \mathbf{f}\left(\bm{0}\right) $ and the definite integral of $\mathbf{u}_0\left(\mathbf{x}\right) $ in $ S_h $.

\begin{lemma} \label{lem_IP_est_f0_int_Sh_u0}
	Consider that $ \mathbf{u}_0\left(\mathbf{x}\right) $ is the CGO solution defined in \eqref{def_CGOs_u0}. Let the convex sectorial corner $ S_h\subset D_R $ be defined as in \eqref{def_inverse_problem_local_sector_Sh_Gammah_Lambdah}. Suppose that the unit vector $ \mathbf{d} $ appearing in \eqref{def_IP_zeta_eta} satisfies \eqref{ineq_IP_dcodtx}. Then for $ \mathbf{f}\left(\bm{0}\right) \in \mathbb{R}^2 $, the estimate
	\begin{align*} 
		\left| \mathbf{f}\left(\bm{0}\right) \cdot \int_{S_h} \mathbf{u}_0\left(\mathbf{x}\right) \mathrm{d}\mathbf{x} \right| \geqslant & \left| \mathbf{f}\left(\bm{0}\right) \right| \left[\left(\theta_M-\theta_m\right)\left(1-\mathcal{O}\left(\frac{1}{\tau^2}\right)\right) \right. \\
		& \left.\left(\frac{1}{3 \tau^2}\left|\mathbf{d}\cdot\hat{\mathbf{x}} \left(\theta_{\xi_1}\right)\right| \left|\mathbf{d}^{\perp}\cdot\hat{\mathbf{x}} \left(\theta_{\xi_1}\right)\right| - \frac{2}{\delta\tau} e^{-\frac{h}{2} \delta \tau}\right) - \frac{c}{\tau}e^{-r_{\varepsilon} \delta \tau}\right] 
	\end{align*}
	holds. Here, $ \tau $ is sufficiently large, $ \hat{\mathbf{x}} \in \Lambda_1 $, $ \theta_{\xi_1} \in \left(\theta_m, \theta_M\right) $, $ \mathbf{d}^{\perp}\cdot\hat{\mathbf{x}} \left(\theta_{\xi_1}\right) \neq 0 $, $ r_{\varepsilon}\in\left(0,h\right) $, and $ c $ is a constant independent of $ \tau $.
\end{lemma}

\begin{proof}
	From the definition of $ \mathbf{u}_0\left(\mathbf{x}\right) $ in \eqref{def_CGOs_u0}, we know
	\begin{align*} 
		\left| \mathbf{f}\left(\bm{0}\right) \cdot \int_{S_h} \mathbf{u}_0\left(\mathbf{x}\right) \mathrm{d}\mathbf{x} \right| \geqslant \left| \mathbf{f}\left(\bm{0}\right) \cdot \bm{\eta} \right| \left|\int_{S_h} e^{\bm{\zeta} \cdot \mathbf{x}} \mathrm{d}\mathbf{x} \right| - \left| \mathbf{f}\left(\bm{0}\right) \right| \left|\int_{S_h} e^{\bm{\zeta} \cdot \mathbf{x}} \mathbf{R}\left(\mathbf{x}\right) \mathrm{d}\mathbf{x} \right| .
	\end{align*}
	Denote $ \mathbf{d} $ and $ \mathbf{d}^\perp $ fulfilling \eqref{ineq_IP_dcodtx} as
	\begin{align} \label{eq_IP_ddperp_sincosphi}
		\mathbf{d}=\begin{pmatrix}
			\cos\theta_d \\
			\sin\theta_d
		\end{pmatrix} ,\quad
		\mathbf{d}^\perp=\begin{pmatrix}
			-\sin\theta_d \\
			\cos\theta_d
		\end{pmatrix}, \quad\theta_d\in\left(0,2\pi\right).
	\end{align}
	Then, when $ \tau $ is large enough, by virtue of \eqref{def_IP_zeta_eta} and the Taylor expansion, we have 
	\begin{align} \label{eq_IP_eta_Taylorexpansion_e1_-i1}
		\bm{\eta} = - \mathrm{i}  \mathbf{d} + \mathbf{d}^\perp - \mathcal{O}\left(\frac{1}{\tau^2} \mathbf{d}\right) = e^{-\mathrm{i}\theta_d} \mathbf{e}_1 - \mathcal{O}\left(\frac{1}{\tau^2} \mathbf{d}\right),\quad
		\mathbf{e}_1:=\begin{pmatrix}
			- \mathrm{i} \\
			1
		\end{pmatrix} .
	\end{align}
	Using \eqref{eq_IP_eta_Taylorexpansion_e1_-i1}, we obtain
	\begin{align*} 
		\left| \mathbf{f}\left(\bm{0}\right) \cdot \bm{\eta} \right| & \geqslant \left| e^{-\mathrm{i}\theta_d} \mathbf{f}\left(\bm{0}\right) \cdot \mathbf{e}_1 \right| - \left| \mathcal{O}\left(\frac{1}{\tau^2} \mathbf{f}\left(\bm{0}\right) \cdot\mathbf{d}\right) \right| \notag \\
		& \geqslant \left| -\mathrm{i} f_1\left(\bm{0}\right) + f_2\left(\bm{0}\right) \right| - \mathcal{O}\left(\frac{1}{\tau^2}\right) \left| \mathbf{f}\left(\bm{0}\right) \right| = \left(1-\mathcal{O}\left(\frac{1}{\tau^2}\right) \right) \left| \mathbf{f}\left(\bm{0}\right) \right| .
	\end{align*}
	When $ \tau $ is large enough, from the polar coordinate transformation, Proposition \ref{prop_inverse_problem_integral_re}, and the integral mean value theorem, it yields that
	\begin{align*} 
		\left|\int_{S_h} e^{\bm{\zeta} \cdot \mathbf{x}} \mathrm{d}\mathbf{x}\right| & = \left| \int_{\theta_m}^{\theta_M} \int_0^h r e^{r \left( \tau \mathbf{d} + \text{i} \sqrt{\tau^2 + k_{\mathrm{s}}^2} \mathbf{d}^{\perp} \right) \cdot \hat{\mathbf{x}}\left(\theta\right)} \mathrm{d}r \mathrm{d}\theta\right| \notag \\
		& \geqslant \left| \int_{\theta_m}^{\theta_M}  \frac{1}{\left( \left( \tau \mathbf{d} + \text{i} \sqrt{\tau^2+k_{\mathrm{s}}^2} \mathbf{d}^{\perp} \right) \cdot \hat{\mathbf{x}}\left(\theta\right) \right)^2} \mathrm{d}\theta\right| - \int_{\theta_m}^{\theta_M} \left| I_R\left(\theta\right)\right| \mathrm{d}\theta \notag \\
		& \geqslant \frac{\theta_M-\theta_m}{3\tau^2} \left|\mathbf{d}\cdot\hat{\mathbf{x}} \left(\theta_{\xi_1}\right)\right| \left|\mathbf{d}^{\perp}\cdot\hat{\mathbf{x}} \left(\theta_{\xi_1}\right)\right| - \frac{2\left(\theta_M-\theta_m\right)}{- \tau \mathbf{d} \cdot \hat{\mathbf{x}} \left(\theta_{\xi_2}\right)} e^{\frac{h}{2} \tau \mathbf{d} \cdot \hat{\mathbf{x}} \left(\theta_{\xi_2}\right)} \notag \\
		& \geqslant \left(\theta_M-\theta_m\right)\left(\frac{1}{3 \tau^2}\left|\mathbf{d}\cdot\hat{\mathbf{x}} \left(\theta_{\xi_1}\right)\right| \left|\mathbf{d}^{\perp}\cdot\hat{\mathbf{x}} \left(\theta_{\xi_1}\right)\right| - \frac{2}{\delta\tau} e^{-\frac{h}{2} \delta \tau}\right).
	\end{align*}
	Here, $ \hat{\mathbf{x}} \left(\theta\right) = \left(\cos \theta, \sin \theta\right)^{\top}\in \mathbb{S}^1 $, $ I_R \left(\theta\right) = \int_h^{+ \infty} r e^{r \left( \tau \mathbf{d} + \text{i} \sqrt{\tau^2 + k_{\mathrm{s}}^2} \mathbf{d}^{\perp} \right) \cdot \hat{\mathbf{x}} \left(\theta\right)} \mathrm{d}r $, $ \theta_{\xi_1}, \theta_{\xi_2} \in \left(\theta_m, \theta_M\right) $, and $ \mathbf{d}^{\perp}\cdot\hat{\mathbf{x}} \left(\theta_{\xi_1}\right) \neq 0 $. Using the integral mean value theorem, \eqref{ineq_IP_ezetax_Sh}, and \eqref{est_inverse_problem_CGO}, it turns out that
	\begin{align*} 
		\left|\int_{S_h} e^{\bm{\zeta} \cdot \mathbf{x}} \mathbf{R}\left(\mathbf{x}\right) \mathrm{d}\mathbf{x} \right| \leqslant \int_{S_h} e^{-\delta \left|\mathbf{x}\right| \tau} \left|\mathbf{R}\left(\mathbf{x}\right)\right| \mathrm{d}\mathbf{x} \leqslant c e^{-\delta r_{\varepsilon} \tau} \left\| \mathbf{R}\left(\mathbf{x}\right) \right\|_{L^2\left(S_h\right)} \leqslant \frac{c}{\tau} e^{-\delta r_{\varepsilon} \tau} ,
	\end{align*}
	where $ r_{\varepsilon}\in\left(0,h\right) $, and $ c $ is a constant independent of $ \tau $. 
	
	The proof is complete. 
\end{proof}

We prove an auxiliary theorem as follows, which plays a role in the characterization of the non-radiating elastic source at the convex corner point. It indicates that, if the vector field $ \mathbf{v}\left(\mathbf{x}\right) $ and $ T_{\bm{\nu}} \mathbf{v}\left(\mathbf{x}\right) $ defined in \eqref{model_part_v_Sh} vanish on the edges near the corner, then $ \mathbf{f}\left(\mathbf{x}\right) $ equals $ \mathbf{0} $ at the corner point, and its gradient $ \nabla\mathbf{f}\left(\mathbf{x}\right) $ satisfies algebraic equations associated with the angles defining the corners.

\begin{theorem} \label{thm_inverse_problem_Sh_f0=0_nablaf0=0}
	Let $ \lambda $ and $ \mu $ fulfill \eqref{ineq_introduction_strong_convexity_condition} with $ \omega > 0 $. Assume that the positive function $ \rho\left(\mathbf{x}\right)\in C^2\left(\mathbb{R}^2\right) $ and $ \rho\left(\mathbf{x}\right)=1 $ for $ \mathbf{x}\in\mathbb{R}^2 \backslash D_R $. For the sector $ S_{\mathbf{x}_0,h} $ and its boundaries $ \Gamma_{\mathbf{x}_0,h}^\pm $ defined in \eqref{def_inverse_problem_local_sector_Sh_Gammah_Lambdah}, we suppose that $ S_{\mathbf{x}_0,h}\subset D_R $. \textcolor{black}{Let $ \mathbf{v}\left(\mathbf{x}\right) \in H^1\left(S_{\mathbf{x}_0,h}\right) $ satisfy} \eqref{model_part_v_Sh}. If $\mathbf{f}\left(\mathbf{x}\right) \in C^{\alpha} \left(\overline{S_{\mathbf{x}_0,h}}\right)$ with $\alpha\in (0,1)$,  then \begin{align*} 
\mathbf{f} \left(\mathbf{x}_0\right) = \bm{0}. 
	\end{align*}
If $\mathbf{f}\left(\mathbf{x}\right) \in C^{1,\alpha} \left(\overline{S_{\mathbf{x}_0,h}}\right)$ with $\alpha\in (0,1)$,  then 
	\begin{equation} \label{eq_IP_gradientf}
\mathbf{f} \left(\mathbf{x}_0\right) = \bm{0},\	\begin{aligned}
			\nabla\mathbf{f}\left(\mathbf{x}_0\right)
			\begin{pmatrix}
				A \\ 
				B
			\end{pmatrix}
			+
			\begin{pmatrix}
				0 & -1 \\
				1 & 0
			\end{pmatrix}
			\nabla\mathbf{f}\left(\mathbf{x}_0\right)
			\begin{pmatrix}
				C \\ 
				D
			\end{pmatrix}
			=\mathbf{0},
		\end{aligned}
	\end{equation}
	with $ A $, $ B $, $ C $, and $ D $ defined in \eqref{def_thm_2_ABCD}. 
\end{theorem}

\begin{proof} 
	Without loss of generality, we suppose that $ \mathbf{x}_0 = \bm{0} $, as the operator $ \mathcal{L} $ is invariant under rigid motion. In the following, we deal with the case that $\mathbf{f}\left(\mathbf{x}\right) \in C^{1,\alpha} \left(\overline{S_{\mathbf{0},h}}\right)$. The case for $\mathbf{f}\left(\mathbf{x}\right) \in C^{\alpha} \left(\overline{S_{\mathbf{0},h}}\right)$ can be proved in a similar way. By virtue of Lemma \ref{lem_CGOs} and the first Betti's formula, we have \eqref{eq_fu0=u0Tnuv_original} and \eqref{eq_IP_fexpansion_Bettiformula}. Then it follows that
	\begin{align*} 
		\left| \mathbf{f} \left(\bm{0}\right) \cdot \int_{S_h} \mathbf{u}_0 \mathrm{d}\mathbf{x} \right|	\leqslant & \left| \int_{\Lambda_h} \mathbf{u}_0 \cdot T_{\bm{\nu}} \mathbf{v} - \mathbf{v} \cdot T_{\bm{\nu}} \mathbf{u}_0 \mathrm{d}\sigma \right| + \left| \int_{S_h} \left( \nabla \mathbf{f} \left(\bm{0}\right) \mathbf{x} \right) \cdot \mathbf{u}_0 \mathrm{d}\mathbf{x} \right| \\
		& + \left| \int_{S_h} \bm{\delta} \mathbf{f} \left(\mathbf{x}\right) \cdot \mathbf{u}_0 \mathrm{d}\mathbf{x} \right|. 
	\end{align*}
	With the help of Lemma \ref{lem_IP_est_int_Lamdah_u0Tv-vTu0}-\ref{lem_IP_est_f0_int_Sh_u0}, as $ \tau $ is sufficiently large, we have  
	\begin{align} \label{ineq_inverse_problem_f0_tau^2}
		\left| \mathbf{f}\left(\bm{0}\right) \right|\left[\left(\theta_M-\theta_m\right)\left(1-\mathcal{O}\left(\frac{1}{\tau^2}\right)\right) \left(\frac{1}{3 \tau^2}\left|\mathbf{d}\cdot\hat{\mathbf{x}} \left(\theta_{\xi_1}\right)\right| \left|\mathbf{d}^{\perp}\cdot\hat{\mathbf{x}} \left(\theta_{\xi_1}\right)\right| - \frac{2}{\delta\tau} e^{-\frac{h}{2} \delta \tau}\right) \right. \notag \\
		\left.- \frac{c}{\tau}e^{-r_{\varepsilon} \delta \tau}\right] \leqslant  c \tau e^{- \delta h \tau} +  \frac{c}{\tau^3} + \frac{c}{\tau^{3+\alpha}}.
	\end{align}
	Here, $ c $ is a constant independent of $ \tau $, $ r_{\varepsilon}\in\left(0,h\right) $, $ \theta_{\xi_1} \in \left(\theta_m, \theta_M\right) $, $ \mathbf{d} $, $ \mathbf{d}^{\perp} $ and $ \hat{\mathbf{x}}\left(\theta_{\xi_1}\right) $ are unit vectors satisfying \eqref{ineq_IP_dcodtx} and $ \mathbf{d}^{\perp}\cdot\hat{\mathbf{x}} \left(\theta_{\xi_1}\right) \neq 0 $. Multiplying $ \tau^2 $ on the both sides of \eqref{ineq_inverse_problem_f0_tau^2}, we obtain 
	\begin{align*} 
		&\left(\theta_M-\theta_m\right)\left(\frac{1}{3 }\left|\mathbf{d}\cdot\hat{\mathbf{x}} \left(\theta_{\xi_1}\right)\right| \left|\mathbf{d}^{\perp}\cdot\hat{\mathbf{x}} \left(\theta_{\xi_1}\right)\right| - \frac{2}{\delta}\tau e^{-\frac{h}{2} \delta \tau}\right) \left| \mathbf{f}\left(\bm{0}\right) \right| \\
		\leqslant &\frac{c}{\tau} +  \frac{c}{\tau^{1+\alpha}} + \mathcal{O}\left(\frac{1}{\tau^2}\right) + c\tau e^{-r_{\varepsilon} \delta \tau} + c \tau^3 e^{- \delta h \tau} .
	\end{align*}
	Letting $ \tau\to+\infty $, it yields that $ \mathbf{f} \left(\bm{0}\right) = \bm{0} $.
	
	Let $ \mathbf{d} $, $ \mathbf{d}^{\perp} $ and $ \mathbf{e}_1 $ be formulated as in \eqref{eq_IP_ddperp_sincosphi} and \eqref{eq_IP_eta_Taylorexpansion_e1_-i1} respectively. When $ \tau $ is sufficiently large, $ \bm{\zeta} $ and $ \bm{\eta} $ appearing in Lemma \ref{lem_CGOs} can be formulated as
	\begin{align}
		\bm{\zeta} = \tau \left( \mathbf{d} + \text{i} \mathbf{d}^{\perp} \right) + \mathcal{O} \left(\frac{1}{\tau}\mathbf{d}^{\perp}\right) = \tau \mathrm{i} e^{-\mathrm{i}\theta_d} \mathbf{e}_1  + \mathcal{O} \left(\frac{1}{\tau}\mathbf{d}^{\perp}\right), \notag
	\end{align}
	and \eqref{eq_IP_eta_Taylorexpansion_e1_-i1}, respectively. \textcolor{black}{Recall $\mathbf{u}_{01}$ and $\mathbf{u}_{02}$ defined in \eqref{def_IP_u01u02}.} Then we have
	\begin{align} \label{eq_intSh_fu01}
		& \int_{S_h} \mathbf{f} \cdot \mathbf{u}_{01} \mathrm{d}\mathbf{x} = \int_{S_h} e^{\bm{\zeta} \cdot \mathbf{x}} \left[\left(\nabla \mathbf{f} \left(\bm{0}\right) \mathbf{x}\right)\cdot\bm{\eta} + \bm{\delta} \mathbf{f}\left(\mathbf{x}\right)\cdot \bm{\eta}\right] \mathrm{d}\mathbf{x} \notag \\
		= & \int_{S_h} e^{\left[\tau \mathrm{i} e^{-\mathrm{i}\theta_d} \mathbf{e}_1 \cdot \mathbf{x} + \mathcal{O} \left(\frac{1}{\tau}\mathbf{d}^{\perp} \cdot \mathbf{x}\right)\right]}\left[ \left(\nabla \mathbf{f} \left(\bm{0}\right) \mathbf{x}\right)\cdot\mathbf{e}_1 e^{-\mathrm{i}\theta_d}\right. + \mathcal{O} \left(\frac{1}{\tau^2}\left(\nabla \mathbf{f} \left(\bm{0}\right) \mathbf{x}\right)\cdot\mathbf{d}\right)\notag \\
		& \left. + \mathcal{O}\left(\left|\mathbf{x}\right|^{1+\alpha}\right)\right] \mathrm{d}\mathbf{x} \notag \\
		= & \int_{\theta_m}^{\theta_M} \int_{0}^{h} r^2 e^{r \tau\left[e^{\mathrm{i}\left(\theta-\theta_d\right)} + \mathcal{O}\left(\frac{1}{\tau^2}\right)\right]}\left[e^{-\mathrm{i}\theta_d}g\left(\theta\right) + \mathcal{O}\left(\frac{1}{\tau^2}\right) + \mathcal{O}\left(r^{\alpha}\right)\right] \mathrm{d}r \mathrm{d}\theta ,
	\end{align}
	where $ g\left(\theta\right) := -\text{i}\left(\frac{\partial f_1}{\partial x_1}\left(\bm{0}\right)\cos\theta + \frac{\partial f_1}{\partial x_2}\left(\bm{0}\right)\sin\theta\right) + \left(\frac{\partial f_2}{\partial x_1}\left(\bm{0}\right)\cos\theta + \frac{\partial f_2}{\partial x_2}\left(\bm{0}\right)\sin\theta\right) $. Due to \eqref{ineq_IP_dcodtx}, for all $ \theta\in\left[\theta_m, \theta_M\right] $, it holds that $ \cos \left(\theta-\theta_d\right) \leqslant -\delta < 0 $. By Proposition \ref{prop_inverse_problem_integral_re}, we obtain
	\begin{align*}
		& \int_{\theta_m}^{\theta_M} \int_{0}^{h} r^2 e^{r \tau\left[e^{\mathrm{i}\left(\theta-\theta_d\right)} + \mathcal{O}\left(\frac{1}{\tau^2}\right)\right]} e^{-\mathrm{i}\theta_d}g\left(\theta\right) \mathrm{d}r \mathrm{d}\theta \\
		= & \int_{\theta_m}^{\theta_M} \frac{\Gamma\left(3\right) }{-\tau^3\left[e^{\text{i}\left(\theta-\theta_d\right)} + \mathcal{O} \left(\frac{1}{\tau^2}\right)\right]^3} e^{-\text{i}\theta_d}g\left(\theta\right) \mathrm{d}\theta - \int_{\theta_m}^{\theta_M} e^{-\text{i}\theta_d} I_R\left(\theta\right) g\left(\theta\right) \mathrm{d}\theta \\
		= & \left(\frac{2e^{2\text{i}\theta_d}}{-\tau^3} + \mathcal{O} \left(\frac{1}{\tau^5}\right)\right)  \int_{\theta_m}^{\theta_M} e^{-3\text{i}\theta} g\left(\theta\right) \mathrm{d}\theta + \mathcal{O} \left( \frac{1}{\tau} e^{-\frac{h}{2}\delta\tau} \right) ,
	\end{align*}
	where $ I_R\left(\theta\right) := \int_{h}^{+\infty} r^2 e^{r \tau\left[e^{\mathrm{i}\left(\theta-\theta_d\right)} + \mathcal{O}\left(\frac{1}{\tau^2}\right)\right]} \mathrm{d}r $, and 
	\begin{align*}
		\left| I_R\left(\theta\right) \right| \leqslant \frac{2}{-\tau\left[\cos\left(\theta-\theta_d\right) + \mathcal{O}\left(\frac{1}{\tau^2}\right)\right]}e^{\frac{h}{2}\tau\left[\cos\left(\theta-\theta_d\right) + \mathcal{O}\left(\frac{1}{\tau^2}\right)\right]} .
	\end{align*}
	Similarly, it can be calculated that
	\begin{align*}
		\int_{\theta_m}^{\theta_M} \int_{0}^{h} r^2 e^{r \tau\left[e^{\mathrm{i}\left(\theta-\theta_d\right)} + \mathcal{O}\left(\frac{1}{\tau^2}\right)\right]} \mathcal{O}\left(\frac{1}{\tau^2}\right) \mathrm{d}r \mathrm{d}\theta = \mathcal{O} \left(\frac{1}{\tau^5}\right) ,\\
		\int_{\theta_m}^{\theta_M} \int_{0}^{h} r^2 e^{r \tau\left[e^{\mathrm{i}\left(\theta-\theta_d\right)} + \mathcal{O}\left(\frac{1}{\tau^2}\right)\right]} \mathcal{O}\left(r^{\alpha}\right) \mathrm{d}r \mathrm{d}\theta = \mathcal{O} \left(\frac{1}{\tau^{3+\alpha}}\right) .
	\end{align*}
	Then, \eqref{eq_intSh_fu01} can be rewritten as
	\begin{align} \label{eq_IP_int_Sh_fu01}
		\int_{S_h} \mathbf{f} \cdot \mathbf{u}_{01} \mathrm{d}\mathbf{x} = & \left(\frac{2e^{2\text{i}\theta_d}}{-\tau^3} + \mathcal{O} \left(\frac{1}{\tau^5}\right)\right)  \int_{\theta_m}^{\theta_M} e^{-3\text{i}\theta} g\left(\theta\right) \mathrm{d}\theta + \mathcal{O} \left( \frac{1}{\tau} e^{-\frac{h}{2}\delta\tau} \right) + \mathcal{O} \left(\frac{1}{\tau^5}\right)\notag \\
		&  + \mathcal{O} \left(\frac{1}{\tau^{3+\alpha}}\right) .
	\end{align}
	By \eqref{eq_fu0=u0Tnuv_original}, Lemma \ref{lem_IP_est_int_Lamdah_u0Tv-vTu0}, and integral mean value theorem, as $ \tau $ is large enough, we arrive that
	\begin{align} \label{ineq_IP_int_Sh_fu01}
		\left|\int_{S_h} \mathbf{f} \cdot \mathbf{u}_{01} \mathrm{d}\mathbf{x}\right| & \leqslant \left|\int_{S_h} \mathbf{f} \cdot \mathbf{u}_{0} \mathrm{d}\mathbf{x}\right| + \left|\int_{S_h} \mathbf{f} \cdot \mathbf{u}_{02} \mathrm{d}\mathbf{x}\right| \notag \\
		& \leqslant \left|\int_{\Lambda_h} \mathbf{u}_0 \cdot T_{\bm{\nu}} \mathbf{v} - \mathbf{v} \cdot T_{\bm{\nu}} \mathbf{u}_0 \mathrm{d}\sigma\right| + ce^{-\delta\left|\mathbf{x}_{\varepsilon}\right|\tau} \left|\mathbf{f}\left(\mathbf{x}_{\varepsilon}\right)\right| \left\|\mathbf{R}\left(\mathbf{x}\right)\right\|_{L^2\left(S_h\right)} \notag \\
		& \leqslant c\tau e^{- \delta h \tau} +  \frac{c}{\tau}e^{-\delta\left|\mathbf{x}_{\varepsilon}\right|\tau} ,
	\end{align}
	where $ \mathbf{x}_{\varepsilon} \in S_h $, $ \left|\mathbf{x}_{\varepsilon}\right|\in\left(0,h\right) $, and $ c $ is a constant independent of $ \tau $. Combining \eqref{eq_IP_int_Sh_fu01} and \eqref{ineq_IP_int_Sh_fu01}, it can be formulated that
	\begin{align} \label{eq_IP_gtau^3}
		& \left|\frac{2e^{2\text{i}\theta_d}}{-\tau^3}  \int_{\theta_m}^{\theta_M} e^{-3\text{i}\theta} g\left(\theta\right) \mathrm{d}\theta + \mathcal{O} \left( \frac{1}{\tau} e^{-\frac{h}{2}\delta\tau} \right) + \mathcal{O} \left(\frac{1}{\tau^5}\right) + \mathcal{O} \left(\frac{1}{\tau^{3+\alpha}}\right)\right| \notag \\
		\leqslant & c\tau e^{- \delta h \tau} + \frac{c}{\tau}e^{-\delta\left|\mathbf{x}_{\varepsilon}\right|\tau} .
	\end{align}
	Multiply $ \tau^3 $ on the both ends of \eqref{eq_IP_gtau^3} and let $ \tau \to +\infty $. It yields that
	\begin{align} \label{eq_inverse_problem_int_eg=0}
		\int_{\theta_m}^{\theta_M} e^{-3\text{i}\theta} g\left(\theta\right) \mathrm{d}\theta = 0 . 
	\end{align}
	From the real and imaginary parts of \eqref{eq_inverse_problem_int_eg=0}, we arrive at \eqref{eq_IP_gradientf} with $ A $, $ B $, $ C $, and $ D $ defined in \eqref{def_thm_2_ABCD}.
	
	The proof is complete.
\end{proof}

With Theorem \ref{thm_inverse_problem_Sh_f0=0_nablaf0=0} in hand, we now present the proof of Theorem \ref{thm_inverse_problem_non-radiating}.

\begin{proof} [Proof of Theorem \ref{thm_inverse_problem_non-radiating}]
	Since $ \mathbf{f} $ is a non-radiating source, its caused far-field pattern $ \mathbf{u}^{\infty}\left(\hat{\mathbf{x}}\right) = \bm{0} $ for all $ \hat{\mathbf{x}}\in\mathbb{S}^1 $. By virtue of Rellich's lemma and the unique continuation principle, we have
	\begin{align*}
		\mathbf{u}\left(\mathbf{x}\right) = \bm{0}, \quad \mathbf{x}\in\mathbb{R}^2 \backslash \overline{\Omega}.
	\end{align*}
	Then from the trace theorem and the transmission conditions, one has
	\begin{align*} 
		\begin{cases}
			\mathcal{L} \mathbf{u} + \omega^2 \rho\left(\mathbf{x}\right) \mathbf{u} = \mathbf{f} & \text{in } S_{\mathbf{x}_0,h} , \\
			\mathbf{u} = \bm{0}, \ T_{\bm{\nu}} \mathbf{u} = \bm{0} & \text {on } \Gamma_{\mathbf{x}_0,h}^\pm .
		\end{cases}
	\end{align*}
	With the help of Theorem \ref{thm_inverse_problem_Sh_f0=0_nablaf0=0}, by considering two kinds of regularity conditions of $\mathbf f$ around the corner, one can readily finish the proof.
\end{proof}

We now prove the local unique determination of the shape of the source's support.

\begin{proof}[Proof of Theorem \ref{thm_inverse_problem_local_uniqueness}]	
	
	\textcolor{black}{We prove this theorem by contradiction. Without loss of generality, we suppose that there exists a convex sectorial corner $ S_{\mathbf{x}_0,h} $ of $\Omega\Delta\Omega^{\prime}$, as defined in Definition \ref{def_admissible_elastic_source_local}, satisfying 
		\begin{align*}
			S_{\mathbf{x}_0,h}\subset\Omega, \quad \overline{S_{\mathbf{x}_0,h}}\cap\overline{\Omega^{\prime}}=\emptyset.
	\end{align*}
	And for each point $\mathbf{x}$ of the boundaries $ \Gamma_{\mathbf{x}_0,h}^\pm $, there exists an unbounded path $\gamma\subset\mathbb{R}^2\backslash\left(\overline{\Omega}\cup\overline{\Omega^\prime}\right)$ connecting $\mathbf{x}$ to infinity. Let $ \mathbf{u} $ and $ \mathbf{u}^{\prime} $ be the elastic fields of the system \eqref{model_global_non-nested_R2} corresponding to $ \left(\Omega; \mathbf{f}\right) $ and $ \left(\Omega^{\prime}; \mathbf{f}^{\prime}\right) $ respectively. Since $ \mathbf{u}_{\infty}\left(\hat{\mathbf{x}}\right) =\mathbf{u}_{\infty}^{\prime} \left(\hat{\mathbf{x}}\right) $ for all $ \hat{\mathbf{x}}\in\mathbb{S}^1 $, utilizing Rellich's lemma and the unique continuation principle, we have
	\begin{align*}
		\mathbf{u}\left(\mathbf{x}\right) = \mathbf{u}^{\prime}\left(\mathbf{x}\right), \quad \mathbf{x}\in D_{\mathbf{x}_0,h}\backslash\overline{S_{\mathbf{x}_0,h}}.
	\end{align*}}
	With the trace theorem and the transmission conditions, the following equations hold 
	\begin{align*} 
		\begin{cases}
			\mathcal{L} \mathbf{u} + \omega^2 \rho\left(\mathbf{x}\right) \mathbf{u} = \mathbf{f} & \text{in } S_{\mathbf{x}_0,h} , \\
			\mathcal{L} \mathbf{u}^{\prime} + \omega^2 \rho\left(\mathbf{x}\right) \mathbf{u}^{\prime} = \bm{0} & \text{in } S_{\mathbf{x}_0,h} , \\
			\mathbf{u} = \mathbf{u}^{\prime}, \ T_{\bm{\nu}} \mathbf{u}=T_{\bm{\nu}} \mathbf{u}^{\prime} & \text {on } \Gamma_{\mathbf{x}_0,h}^{\pm} .
		\end{cases}
	\end{align*}
	Denote $ \mathbf{v} = \mathbf{u} - \mathbf{u}^{\prime} $. It follows that
	\begin{align*} 
		\begin{cases}
			\mathcal{L} \mathbf{v} + \omega^2 \rho\left(\mathbf{x}\right) \mathbf{v} = \mathbf{f} & \text{in } S_{\mathbf{x}_0,h} , \\
			\mathbf{v} = \bm{0}, \ T_{\bm{\nu}} \mathbf{v} = \bm{0} & \text {on } \Gamma_{\mathbf{x}_0,h}^{\pm} .
		\end{cases}
	\end{align*}
	Using Theorem \ref{thm_inverse_problem_Sh_f0=0_nablaf0=0}, if $\mathbf f$ fulfills the admissible condition (2.1) in Definition \ref{def_admissible_elastic_source_local}, then one readily arrive at $ \mathbf{f} \left(\mathbf{x}_0\right) = \bm{0} $, where we get the contradiction. Similarly, if $\mathbf f$ fulfills the admissible condition (2.2) in Definition \ref{def_admissible_elastic_source_local}, then \eqref{eq_IP_gradientf} is satisfied,  which is contradict with the definition of the admissible elastic sources. 
	
	The proof is complete.
\end{proof}

Based on the local uniqueness results, we now proceed to prove the global results. 

\begin{proof}[Proof of Theorem \ref{thm_inverse_problem_global_uniqueness}]
	
	Using Theorem \ref{thm_inverse_problem_local_uniqueness} at each corner, one can readily arrive that $ \Omega \Delta \Omega^{\prime} $ cannot contain any sectorial corner. Since $ \Omega $ and $ \Omega^{\prime} $ are convex polygons, it follows that $ \Omega = \Omega^{\prime} $. At each corner $ S_{\mathbf{x}_{i},h} $, $ h\in\mathbb{R}_+ $, as stated in the proof of Theorem \ref{thm_inverse_problem_local_uniqueness}, we have 
	\begin{align*} 
		\begin{cases}
			\mathcal{L} \mathbf{u} + \omega^2 \rho\left(\mathbf{x}\right) \mathbf{u} = \mathbf{f} & \text{in } S_{\mathbf{x}_{i},h} , \\
			\mathcal{L} \mathbf{u}^{\prime} + \omega^2 \rho\left(\mathbf{x}\right) \mathbf{u}^{\prime} = \mathbf{f}^{\prime} & \text{in } S_{\mathbf{x}_{i},h} , \\
			\mathbf{u} = \mathbf{u}^{\prime}, \ T_{\bm{\nu}} \mathbf{u}=T_{\bm{\nu}} \mathbf{u}^{\prime} & \text {on } \Gamma_{\mathbf{x}_{i},h}^{\pm} .
		\end{cases}
	\end{align*}
	Letting $ \mathbf{w} = \mathbf{u} - \mathbf{u}^{\prime} $, it can be written as
	\begin{align*} 
		\begin{cases}
			\mathcal{L} \mathbf{w} + \omega^2 \rho\left(\mathbf{x}\right) \mathbf{v} = \mathbf{f}-\mathbf{f}^{\prime} & \text{in } S_{\mathbf{x}_{i},h} , \\
			\mathbf{w} = \bm{0}, \ T_{\bm{\nu}} \mathbf{w} = \bm{0} & \text {on } \Gamma_{\mathbf{x}_{i},h}^{\pm} .
		\end{cases}
	\end{align*}
	Then Theorem \ref{thm_inverse_problem_Sh_f0=0_nablaf0=0} implies that $ \mathbf{f} $ and $ \mathbf{f}^{\prime} $ satisfy \eqref{eq_thm_3_f=fprime} and \eqref{eq_thm_3_gradientf=gradientfprime} at each corner.  
	
	The proof is complete.
\end{proof}

\appendix 
\section{Proof of Lemma \ref{lem_CGOs}} \label{section_appendix} 
In this section, we aim to prove Lemma \ref{lem_CGOs}. Namely,  we need to demonstrate  the existence of the CGO solution $ \mathbf{u}_0\left(\mathbf{x}\right)$ given by \eqref{def_CGOs_u0}  to the Navier equation \eqref{eq_CGOs_Navier_equation_no_source}, with the estimates \eqref{est_inverse_problem_CGO}. Here, $ \bm{\zeta} $ and $ \bm{\eta} $ are defined in \eqref{def_IP_zeta_eta}, $ \tau $ is large enough, and $ c $ is a constant independent of $ \tau $.

We rewrite \eqref{eq_CGOs_Navier_equation_no_source} as an equivalent modified Lippmann-Schwinger equation \eqref{eq_CGOs_modified_Lippmann-Schwinger_equation}, which involves the modified fundamental solution $ \bm{\Psi}^{\bm{\zeta}}\left(\mathbf{x}\right) $ defined by \eqref{def_CGOs_modified_fundamental_solution_Omegazeta} to the operator $ \mathcal{L}+\omega^2 $. Introducing $ \mathbf{R}\left(\mathbf{x}\right) $ and $ v\left(\mathbf{x}\right) $ formulated in \eqref{def_CGOs_R_v}, we transform \eqref{eq_CGOs_modified_Lippmann-Schwinger_equation} into the equivalent system \eqref{eq_CGOs_Rv=S+T}. It can be concluded that  the existence of  $ \mathbf{u}_0\left(\mathbf{x}\right) $  is equivalent to show the existence of $\mathbf{R}\left(\mathbf{x}\right)$. By showing that the operator $ \mathcal{I} - \mathcal{T} $ in \eqref{eq_CGOs_Rv=S+T} is invertible, we obtain the existence of $\mathbf{R}\left(\mathbf{x}\right)$. The estimations can be proved with the help of the Fourier series approach.

The construction of the modified fundamental solution $ \bm{\Psi}^{\bm{\zeta}}\left(\mathbf{x}\right) $ and its related estimates depend on the modified fundamental solution $ h_{\bm{\xi}}\left(\mathbf{x}\right) $ to the operator $ \Delta + k^2 $, $ k\in\mathbb{R}_+ $. Let $ 0 < R < R_1 < R^{\prime} $. We set the square $ S := \left(-R^{\prime}, R^{\prime}\right)^2 $ and the disk $ D_{R^{\prime}}:=\left\{\mathbf{x} \in \mathbb{R}^2:\left| \mathbf{x} \right|< R^{\prime}\right\} $. Define the grid 
$$ G := \left\{ \bm{\alpha} = \left(\alpha_1, \alpha_2\right)^{\top} \in \mathbb{R}^2: \frac{R^{\prime}}{\pi} \alpha_1-\frac{1}{2} \in \mathbb{Z}, \frac{R^{\prime}}{\pi} \alpha_2 \in \mathbb{Z}\right\}.$$

Then the functions $ e_{\bm{\alpha}} \left(\mathbf{x}\right) := \frac{1}{2 R^{\prime}}e^{\mathrm{i} \bm{\alpha} \cdot \mathbf{x}} $, $ \bm{\alpha} \in G $, form an orthonormal basis in $ L^2\left(S\right) $. For $ f\left(\mathbf{x}\right) \in L^2\left(S\right) $, we denote the Fourier coefficients $ \hat{f}\left(\bm{\alpha}\right) := \int_{S}f\left(\mathbf{x}\right) e_{\bm{\alpha}} \left(\mathbf{x}\right) \mathrm{d} \mathbf{x}$, $ \bm{\alpha} \in G $. For $ k, \tau > 0 $, and
\begin{equation}\label{eq:bmfxi}
\bm{\xi} = \left(-\tau, \mathrm{i}\sqrt{\tau^2+k^2} \right)^{\top},
\end{equation}
 we set
\begin{align} \label{eq_CGOs_g_xi_Psi_xi}
	g_{\bm{\xi}} \left(\mathbf{x}\right) := \sum_{\bm{\alpha} \in G} \frac{1}{\bm{\alpha} \cdot \bm{\alpha} - 2\mathrm{i} \bm{\xi} \cdot \bm{\alpha}} e_{\bm{\alpha}} \left(\mathbf{x}\right), \quad h_{\bm{\xi}} \left(\mathbf{x}\right) := e^{ \bm{\xi} \cdot \mathbf{x}} g_{\bm{\xi}}\left(\mathbf{x}\right).
\end{align}
Similar to the three-dimensional case \cite{hahner2002uniqueness}, we can verify that $ h_{\bm{\xi}}\left(\mathbf{x}\right) $ is a fundamental solution for $ \Delta+k^2 $ in $ D_{2R^{\prime}}:=\left\{\mathbf{x} \in \mathbb{R}^2:\left| \mathbf{x} \right|< 2R^{\prime}\right\} $. Set $ \tilde{g}_{\bm{\xi}} \left(\mathbf{x}\right) = h_{\bm{\xi}} \left(\mathbf{x}\right) - \frac{\mathrm{i}}{4} H_0^{\left(1\right)} \left(k\left|\mathbf{x}\right|\right) $, where $ H_0^{\left(1\right)}\left(\mathbf{x}\right) $ is the Hankel function of the first kind of order zero. Then it follows that $ \tilde{g}_{\bm{\xi}}\left(\mathbf{x}\right) \in C^{\infty} \left(D_{2R^{\prime}}\right) $ satisfies $ \Delta \tilde{g}_{\bm{\xi}} + k^2 \tilde{g}_{\bm{\xi}} = 0 $.

In order to prove the existence of CGO solution $ \mathbf{u}_0\left(\mathbf{x}\right) $ and obtain the estimates \eqref{est_inverse_problem_CGO} about the residual term $ \mathbf{R}\left(\mathbf{x}\right) $, we need the following lemma, which defines the potential operator and gives related estimations.

\begin{lemma} \label{lem_CGOs_fundamental_solution_and_integral_estimation_gphi}
	Let $ \bm{\xi}$ be defined by \eqref{eq:bmfxi} with $ k,\tau > 0 $. Assume that $ \tau $ is large enough. 
	Define the potential operator $ \mathcal{P} : L^2\left(D_{R_1}\right)\to H^2\left(D_{R_1}\right) $ by
	\begin{align} \label{def_CGOs_operator_P}
		\mathcal{P}\left(f\right)\left(\mathbf{x}\right) := \int_{D_{R_1}} g_{\bm{\xi}} \left(\mathbf{x} - \mathbf{y}\right) f \left(\mathbf{y}\right) \mathrm{d} \mathbf{y},
	\end{align}
	where $ g_{\bm{\xi}}\left(\mathbf{x}\right) $ is given by \eqref{eq_CGOs_g_xi_Psi_xi}. Then for $ f \in L^2\left(D_{R_1}\right) $, it holds that
	\begin{align} 
		\label{est_CGOs_integral_gf_L^2}
		&\left\| \mathcal{P}\left(f\right)  \right\|_{L^2\left(D_{R_1}\right)} \leqslant \frac{c}{\tau} \left\| f \right\|_{L^2\left(D_{R_1}\right)}, \\
		\label{est_CGOs_gradient_integral_gf_L^2}
		&\left\| \nabla \mathcal{P}\left(f\right) \right\|_{L^2\left(D_{R_1}\right)} \leqslant c \left\| f \right\|_{L^2\left(D_{R_1}\right)}, \\
		\label{est_CGOs_gradient_gradient_integral_gf_L^2}
		&\left\| \nabla^2 \mathcal{P}\left(f\right) \right\|_{L^2\left(D_{R_1}\right)} \leqslant c \tau \left\| f \right\|_{L^2\left(D_{R_1}\right)} .
	\end{align}
	Here, $ c $ is a constant depending on $ k $, $ R_1 $ and  $ R^{\prime} $.
\end{lemma}

\begin{proof}
	Notice that for $ \bm{\alpha} \in G $, one has
	\begin{align} \label{ineq_CGOs_alphaalpha-2ixialpha}
		\left| \bm{\alpha} \cdot \bm{\alpha} - 2 \mathrm{i} \bm{\xi} \cdot \bm{\alpha} \right| \geqslant \left| \mathrm{Im}\left(\bm{\alpha} \cdot \bm{\alpha} - 2 \mathrm{i} \bm{\xi} \cdot \bm{\alpha}\right)\right| = 2 \tau \left| \alpha_1 \right| \geqslant \pi \tau/R^{\prime}.
	\end{align}
	When $ f \in C^{\infty}_0 \left(D_{R_1}\right) $, it yields that 
	\begin{align*}
		\left\| \int_{D_{R_1}} g_{\bm{\xi}} \left(\mathbf{x} - \mathbf{y}\right) f \left(\mathbf{y}\right) \mathrm{d} \mathbf{y} \right\|^2_{L^2\left(D_{R_1}\right)} \leqslant \sum_{\bm{\alpha} \in G} \frac{\left|\hat{f}\left(\bm{\alpha}\right)\right|^2}{\left|\bm{\alpha} \cdot \bm{\alpha} - 2\mathrm{i} \bm{\xi} \cdot \bm{\alpha}\right|^2} \leqslant \frac{c}{\tau^2} \left\| f \right\|^2_{L^2\left(D_{R_1}\right)},
	\end{align*}
	where $ c $ is a constant independent of $ \tau $. Since $ C^{\infty}_0 \left(D_{R_1}\right) $ is dense in $ L^2\left(D_{R_1}\right) $, we obtain \eqref{est_CGOs_integral_gf_L^2}. Furthermore, when $ \tau> 0 $ is sufficiently large, there exists a constant $ c_{k,R^{\prime}} $ depending on $ k $ and $ R^{\prime} $, such that
	\begin{align} \label{ineq_CGOs__alpha^2_alpha^4/|alphaalpha-xialpha|^2}
		\frac{\left| \bm{\alpha} \right|^2}{\left|\bm{\alpha} \cdot \bm{\alpha} - 2\mathrm{i} \bm{\xi} \cdot \bm{\alpha}\right|^2} \leqslant c_{k,R^{\prime}},\quad \frac{\left| \bm{\alpha} \right|^4}{\left|\bm{\alpha} \cdot \bm{\alpha} - 2\mathrm{i} \bm{\xi} \cdot \bm{\alpha}\right|^2} \leqslant c_{k,R^{\prime}} \tau^2,
	\end{align}
	for each $ \bm{\alpha} \in G $. Then, with the help of \eqref{ineq_CGOs__alpha^2_alpha^4/|alphaalpha-xialpha|^2}, a similar strategy can be adopted to demonstrate \eqref{est_CGOs_gradient_integral_gf_L^2} and \eqref{est_CGOs_gradient_gradient_integral_gf_L^2}, which finishes the proof.
\end{proof}

Now we use $ h_{\bm{\zeta}}\left(\mathbf{x}\right) $ to construct the modified fundamental solution $ \bm{\Psi}^{\bm{\zeta}}\left(\mathbf{x}\right) $ defined in \eqref{def_CGOs_modified_fundamental_solution_Omegazeta} to the operator $ \mathcal{L}+\omega^2 $. Recall $ \bm{\zeta} $ given by \eqref{def_IP_zeta_eta}. We define
\begin{align} \label{def_CGOs_xi_tildexi} 
	\bm{\xi} := \left(-\tau, \mathrm{i}\sqrt{\tau^2 + k_{\mathrm{s}}^2}\right)^{\top},\quad \tilde{\bm{\xi}} := \left(-\tau, \mathrm{i}\sqrt{\tau^2+k_{\mathrm{p}}^2}\right)^{\top},
\end{align}
where $ k_{\mathrm{s}} $ and $ k_{\mathrm{p}} $ are from \eqref{eq_introduction_kpks}. Then there exists the unit orthogonal transformation $ \mathbf{Q} \in \mathbb{R}^{2 \times 2} $, such that $ \bm{\zeta} =\mathbf{Q} ^\top  \bm{\xi} $. Introduce
\begin{align} \label{def_CGOs_tilde{zeta}=Q^Ttilde{xi}}
	\tilde{\bm{\zeta}} := \mathbf{Q}^{\top} \tilde{\bm{\xi}}.
\end{align}
We set the modified fundamental solution $ \bm{\Psi}^{\bm{\zeta}}\left(\mathbf{x}\right) = \left(\Psi_{i j}^{\bm{\zeta}}\right)_{ij=1}^2\left(\mathbf{x}\right) $ as
\begin{align} \label{def_CGOs_modified_fundamental_solution_Omegazeta}
	\Psi_{i j}^{\bm{\zeta}} \left(\mathbf{x}\right) = \frac{1}{\mu} \delta_{ij} h_{\bm{\zeta}}\left(\mathbf{x}\right) + \frac{1}{\omega^2} \partial_i \partial_j \left( h_{\bm{\zeta}}\left(\mathbf{x}\right) - h_{\tilde{\bm{\zeta}}}\left(\mathbf{x}\right)  \right), \quad \mathbf{x}\in\mathbb{R}^2, \mathbf{x}\neq\mathbf{0}.
\end{align}
Here, $ h_{\bm{\zeta}}\left(\mathbf{x}\right) $ and $ h_{\tilde{\bm{\zeta}}}\left(\mathbf{x}\right) $ are given by \eqref{eq_CGOs_g_xi_Psi_xi}. By the aid of Helmholtz decomposition and the relationship between Navier operator and Helmholtz operator, we can readily verify that $ \bm{\Psi}^{\bm{\zeta}}\left(\mathbf{x}\right) $ is a fundamental solution to the operator $ \mathcal{L}+\omega^2 $.

Let $ \tilde{\bm{\Psi}}^{\bm{\zeta}}\left(\mathbf{x}\right) = \bm{\Psi}^{\bm{\zeta}}\left(\mathbf{x}\right) - \bm{\Gamma}^\omega\left(\mathbf{x}\right) $, where the Kupradze matrix $ \bm{\Gamma}^{\omega}\left(\mathbf{x}\right) =  \left(\Gamma_{i j}^{\omega}\right)_{ij=1}^2\left(\mathbf{x}\right)  $ (cf. \cite{ammari2015mathematical}) is a fundamental solution for the operator $ \mathcal{L}+\omega^2 $, which can be written as
\begin{align*}
	\Gamma_{i j}^{\omega} \left(\mathbf{x}\right) = -\frac{\mathrm{i}}{4\mu} \delta_{ij} H_0^{\left(1\right)} \left(k_{\mathrm{s}}\left|\mathbf{x}\right|\right) - \frac{\mathrm{i}}{4\omega^2} \partial_i \partial_j \left( H_0^{\left(1\right)} \left(k_{\mathrm{s}}\left|\mathbf{x}\right|\right) - H_0^{\left(1\right)} \left(k_{\mathrm{p}}\left|\mathbf{x}\right|\right) \right),
\end{align*}
with $ \mathbf{x}\in\mathbb{R}^2, \mathbf{x}\neq\mathbf{0} $. Since $ \bm{\Psi}^{\bm{\zeta}}\left(\mathbf{x}\right) $ and $ \bm{\Gamma}^\omega\left(\mathbf{x}\right) $ have the same singular behavior at the same point $ \mathbf{0} $, it is easy to verify that the columns of $ \tilde{\bm{\Psi}}^{\bm{\zeta}}\left(\mathbf{x}\right) $ are $ C^{\infty} $-smooth solutions to $ \mathcal{L} \mathbf{u} + \omega^2 \mathbf{u} = \bm{0} $ in $ D_{R_1} $.

For $ \bm{\eta} $ and $ \bm{\zeta} $ defined in \eqref{def_IP_zeta_eta}, it is easy to verify that $ e^{\bm{\zeta} \cdot \mathbf{x}} \bm{\eta}  $ satisfies $ \mathcal{L} \mathbf{u} + \omega^2 \mathbf{u} = \bm{0} $ in $ \mathbb{R}^2 $. Since $ \mathbf{u}_0\left(\mathbf{x}\right) $ fulfills \eqref{eq_CGOs_Navier_equation_no_source}, by Betti's formula, we obtain the modified Lippmann-Schwinger equation
\begin{align} \label{eq_CGOs_modified_Lippmann-Schwinger_equation}
	\mathbf{u}_0\left(\mathbf{x}\right) = e^{\bm{\zeta}\cdot\mathbf{x}} \bm{\eta} - \omega^2 \int_{D_R} \left(1-\rho\left(\mathbf{y}\right)\right) \bm{\Psi}^{\bm{\zeta}}\left(\mathbf{x}-\mathbf{y}\right) \mathbf{u}_0\left(\mathbf{y}\right) \mathrm{d} \mathbf{y}.
\end{align}
When $ \mathbf{u}_0\left(\mathbf{x}\right) \in C \left(\overline{D_{R_1}}\right) $ satisfies \eqref{eq_CGOs_modified_Lippmann-Schwinger_equation}, by virtual of the mapping properties of volume potentials, we readily derive that $ \mathbf{u}_0\left(\mathbf{x}\right) \in C^2\left(D_{R_1}\right) $ is a solution to $ \mathcal{L}\mathbf{u}_0 + \omega^2 \rho\left(\mathbf{x}\right) \mathbf{u}_0 = \bm{0} $ in $ D_{R_1} $.

Inserting \eqref{def_CGOs_modified_fundamental_solution_Omegazeta} into \eqref{eq_CGOs_modified_Lippmann-Schwinger_equation} and using integration by parts, we derive that
\begin{align} \label{eq_CGOs_substituted_modified_Lippmann-Schwinger_equation}
	\mathbf{u}_0\left(\mathbf{x}\right) = & e^{\bm{\zeta}\cdot\mathbf{x}} \bm{\eta}-\frac{\omega^2}{\mu} \int_{D_R}\left(1-\rho\left(\mathbf{y}\right)\right) h_{\bm{\zeta}}\left(\mathbf{x} - \mathbf{y}\right) \mathbf{u}_0\left(\mathbf{y}\right) \mathrm{d} \mathbf{y} \notag \\
	& - \nabla \int_{D_R}\left(h_{\bm{\zeta}}-h_{\tilde{\bm{\zeta}}}\right)\left(\mathbf{x} - \mathbf{y}\right)\left(\left(1-\rho\right) \nabla \cdot \mathbf{u}_0-\nabla \rho \cdot \mathbf{u}_0\right)\left(\mathbf{y}\right) \mathrm{d} \mathbf{y}.
\end{align}
Noting that $ \mathbf{u}_0\left(\mathbf{x}\right) $ and $ \nabla \cdot \mathbf{u}_0\left(\mathbf{x}\right) $ are both in \eqref{eq_CGOs_substituted_modified_Lippmann-Schwinger_equation}, we take the divergence of both sides of \eqref{eq_CGOs_substituted_modified_Lippmann-Schwinger_equation} and use integration by parts. Then it follows that
\begin{align} \label{eq_CGOs_nabla_substituted_modified_Lippmann-Schwinger_equation}
	\nabla \cdot \mathbf{u}_0\left(\mathbf{x}\right) = - k_{\mathrm{p}}^2 \int_{D_R} h_{\tilde{\bm{\zeta}}}\left(\mathbf{x} - \mathbf{y}\right)((1-\rho) \nabla \cdot \mathbf{u}_0-\nabla \rho \cdot \mathbf{u}_0)\left(\mathbf{y}\right) \mathrm{d} \mathbf{y}.
\end{align}
Introduce 
\begin{align} \label{def_CGOs_R_v} 
	\mathbf{R}\left(\mathbf{x}\right) := e^{-\bm{\zeta} \cdot \mathbf{x}} \mathbf{u}_0\left(\mathbf{x}\right) - \bm{\eta}, \quad v\left(\mathbf{x}\right) := e^{-\bm{\zeta} \cdot \mathbf{x}} \nabla \cdot \mathbf{u}_0\left(\mathbf{x}\right) , \quad \mathbf{x} \in \overline{D_{R_1}}.
\end{align}
Multiplying both sides of \eqref{eq_CGOs_substituted_modified_Lippmann-Schwinger_equation} and \eqref{eq_CGOs_nabla_substituted_modified_Lippmann-Schwinger_equation} by $ e^{-\bm{\zeta} \cdot \mathbf{x}} $ respectively, we obtain 
\begin{align} \label{eq_CGOs_Rv=S+T}
	\begin{pmatrix}
		\mathbf{R} \\
		v
	\end{pmatrix}
	= \mathcal{S}\left(\bm{\eta}\right) + \mathcal{T} 
	\begin{pmatrix}
		\mathbf{R} \\
		v
	\end{pmatrix},
\end{align}
where the operators $ \mathcal{S} : L^2\left(D_{R_1}\right)\to L^2\left(D_{R_1}\right) $ and $ \mathcal{T} : L^2\left(D_{R_1}\right)\to L^2\left(D_{R_1}\right) $ are formulated as 
\begin{align*} 
	\mathcal{S}\left(\bm{\eta}\right) :=
	\begin{pmatrix}
		-\frac{\omega^2}{\mu} \mathcal{P}\left(\left(1-\rho\right) \bm{\eta}\right) - \left(\mathcal{T}_1 + \mathcal{T}_2\right) \left(\nabla \rho \cdot \bm{\eta}\right) \\
		-\mathcal{T}_3 \left( \nabla \rho \cdot \bm{\eta}\right)
	\end{pmatrix},
\end{align*}
\begin{align*} 
	\mathcal{T} \begin{pmatrix}
		\mathbf{R} \\
		v
	\end{pmatrix} := 
	\begin{pmatrix}
		-\frac{\omega^2}{\mu} \mathcal{P}\left(\left(1-\rho\right) \mathbf{R}\right) + \left(\mathcal{T}_1 + \mathcal{T}_2\right) \left(\left(1-\rho\right) v - \nabla \rho \cdot \mathbf{R}\right) \\
		\mathcal{T}_3 \left(\left(1-\rho\right) v - \nabla \rho \cdot \mathbf{R}\right)
	\end{pmatrix}.
\end{align*}
Here $ \mathcal{P} $ is given by \eqref{def_CGOs_operator_P}, and we define the operator $ \mathcal{T}_1 $, $ \mathcal{T}_2 $ and $ \mathcal{T}_3 : L^2\left(D_{R_1}\right)\to L^2\left(D_{R_1}\right) $ by
\begin{align*} 
	&\mathcal{T}_1\left(\varphi\right)\left(\mathbf{x}\right) := - \bm{\zeta} \int_{D_{R_1}} \left(g_{\bm{\zeta}}\left(\mathbf{x} - \mathbf{y}\right) - e^{\left(\tilde{\bm{\zeta}}-\bm{\zeta}\right) \cdot \left(\mathbf{x} - \mathbf{y}\right)} g_{\tilde{\bm{\zeta}}}\left(\mathbf{x} - \mathbf{y}\right)\right) \varphi\left(\mathbf{y}\right) \mathrm{d} \mathbf{y}, \\
	&\mathcal{T}_2\left(\varphi\right)\left(\mathbf{x}\right) := - \nabla \int_{D_{R_1}} \left(g_{\bm{\zeta}}\left(\mathbf{x} - \mathbf{y}\right) - e^{\left(\tilde{\bm{\zeta}}-\bm{\zeta}\right) \cdot \left(\mathbf{x} - \mathbf{y}\right)} g_{\tilde{\bm{\zeta}}}\left(\mathbf{x} - \mathbf{y}\right)\right) \varphi\left(\mathbf{y}\right) \mathrm{d} \mathbf{y}, \\
	&\mathcal{T}_3\left(\varphi\right)\left(\mathbf{x}\right) := - k_{\mathrm{p}}^2 \int_{D_{R_1}} g_{\tilde{\bm{\zeta}}}\left(\mathbf{x} - \mathbf{y}\right) e^{\left(\tilde{\bm{\zeta}}-\bm{\zeta}\right) \cdot \left(\mathbf{x} - \mathbf{y}\right)}\varphi\left(\mathbf{y}\right) \mathrm{d} \mathbf{y},
\end{align*}
where $ \varphi \in L^2 \left(D_{R_1}\right) $. Then the existence of the solution $ \mathbf{u}_0\left(\mathbf{x}\right) $ to the modified Lippmann-Schwinger equation \eqref{eq_CGOs_modified_Lippmann-Schwinger_equation} is equivalently transformed into the existence of solutions $ \mathbf{R}\left(\mathbf{x}\right) $ and $ v\left(\mathbf{x}\right) $ in the system \eqref{eq_CGOs_Rv=S+T}.

We now prove the existence of CGO solution \eqref{def_CGOs_u0} with estimates \eqref{est_inverse_problem_CGO}.

\begin{proof}[Proof of Lemma \ref{lem_CGOs}]

In view of the definition of the operators $\mathcal S$ and $\mathcal T$ given by \eqref{eq_CGOs_Rv=S+T}, we only need to prove the mapping property of $\mathcal T_i$ $(i=1,2,3)$. Namely, we shall show that $\mathcal T_i$ are bounded operators from $L^2(D_{R_1})$ to $L^2(D_{R_1})$. Furthermore, the asymptotic analysis with respect to the parameter $\tau$ for the norms of these three operators are investigated when $\tau \rightarrow \infty$.  
	
	For the system \eqref{eq_CGOs_Rv=S+T}, the integrand of $ \mathcal{T}_1\left(\varphi\right) $ can be split into
	\begin{align} \label{eq_CGOs_split_g-eg}
		& \left(g_{\bm{\zeta}}\left(\mathbf{x} - \mathbf{y}\right) - e^{\left(\tilde{\bm{\zeta}}-\bm{\zeta}\right) \cdot \left(\mathbf{x} - \mathbf{y}\right)} g_{\tilde{\bm{\zeta}}}\left(\mathbf{x} - \mathbf{y}\right)\right) \varphi \left(\mathbf{y}\right) = \left(g_{\bm{\zeta}}\left(\mathbf{x} - \mathbf{y}\right) -  g_{\tilde{\bm{\zeta}}}\left(\mathbf{x} - \mathbf{y}\right)\right) \varphi \left(\mathbf{y}\right) \\
		& + g_{\tilde{\bm{\zeta}}}\left(\mathbf{x} - \mathbf{y}\right) \left(1 - e^{-\left(\tilde{\bm{\zeta}}-\bm{\zeta}\right) \cdot \mathbf{y}} \right) \varphi \left(\mathbf{y}\right) + \left(1 - e^{\left(\tilde{\bm{\zeta}}-\bm{\zeta}\right) \cdot \mathbf{x}} \right) g_{\tilde{\bm{\zeta}}}\left(\mathbf{x} - \mathbf{y}\right) e^{-\left(\tilde{\bm{\zeta}}-\bm{\zeta}\right) \cdot \mathbf{y}} \varphi \left(\mathbf{y}\right). \notag
	\end{align}
	From \eqref{def_CGOs_xi_tildexi}, when $ \tau $ is large enough, it yields that
	\begin{align} \label{ineq_CGOs_tildexi-xi}
		\left|\tilde{\bm{\xi}} - \bm{\xi} \right| = \frac{\left| k_{\mathrm{p}}^2 - k_{\mathrm{s}}^2 \right|}{\sqrt{\tau^2+k_{\mathrm{p}}^2} + \sqrt{\tau^2+k_{\mathrm{s}}^2}} \leqslant \frac{c}{\tau},
	\end{align}
	where $ c $ is a constant independent of $ \tau $. Recall $ \bm{\zeta} $ defined in \eqref{def_IP_zeta_eta} and $ \tilde{\bm{\zeta}} $ formulated in \eqref{def_CGOs_tilde{zeta}=Q^Ttilde{xi}}. For $ \varphi \in C^{\infty}_0 \left(D_{R_1}\right) $, we write $ \psi \left(\mathbf{x}\right) = \varphi \left(\mathbf{Q}^{\top} \mathbf{x}\right) $, with the unit orthogonal transformation $ \mathbf{Q} $ in \eqref{def_CGOs_tilde{zeta}=Q^Ttilde{xi}}. Then it leads that $ g_{\bm{\xi}}\left(\mathbf{x}\right) = g_{\bm{\zeta}}\left(\mathbf{Q}^{\top}\mathbf{x}\right) $ and $ g_{\tilde{\bm{\xi}}}\left(\mathbf{x}\right) = g_{\tilde{\bm{\zeta}}}\left(\mathbf{Q}^{\top}\mathbf{x}\right) $. With the help of \eqref{eq_CGOs_g_xi_Psi_xi}, \eqref{ineq_CGOs_alphaalpha-2ixialpha}, \eqref{ineq_CGOs__alpha^2_alpha^4/|alphaalpha-xialpha|^2}, the invariance of the integrals with respect to orthogonal transformations and Parseval's identity, it follows that
	\begin{align} \label{eq_CGOs_zeta_g-g_varphi}
		& \left\| \bm{\zeta} \int_{D_{R_1}} \left(g_{\bm{\zeta}}\left(\mathbf{x} - \mathbf{y}\right) - g_{\tilde{\bm{\zeta}}}\left(\mathbf{x} - \mathbf{y}\right)\right) \varphi \left(\mathbf{y}\right) \mathrm{d} \mathbf{y} \right\|^2_{L^2\left(D_{R_1}\right)} \notag \\	
		= & \int_{D_{R_1}} \left| \sum_{\bm{\alpha} \in G} \frac{2\mathrm{i}\left( \bm{\xi}-\tilde{\bm{\xi}} \right) \cdot \bm{\alpha} \hat{\psi}\left(\bm{\alpha}\right)}{\left(\bm{\alpha} \cdot \bm{\alpha} - 2\mathrm{i} \bm{\xi} \cdot \bm{\alpha}\right)\left(\bm{\alpha} \cdot \bm{\alpha} - 2\mathrm{i} \tilde{\bm{\xi}} \cdot \bm{\alpha}\right)} \bm{\xi} e_{\bm{\alpha}} \left(\mathbf{x}\right) \right|^2 \mathrm{d} \mathbf{x} \notag \\
		\leqslant & \sum_{\bm{\alpha} \in G} \frac{c}{\tau^2} \left|\hat{\psi}\left(\bm{\alpha}\right)\right|^2 = \frac{c}{\tau^2} \left\| \varphi \right\|^2_{L^2\left(D_{R_1}\right)} ,
	\end{align}
	where $ \tau $ is large enough and $ c $ is independent of $ \tau $. Similarly, by virtue of \eqref{est_CGOs_integral_gf_L^2}, 
	\begin{align} \label{ineq_CGOs_sup_1-e}
		\sup_{\mathbf{x} \in D_{R_1}} \left| 1 - e^{-\left(\tilde{\bm{\zeta}}-\bm{\zeta}\right) \cdot \mathbf{x}} \right| \leqslant \frac{c}{\tau}, 
	\end{align}	
	and the uniform boundedness of $ \left| e^{-\left(\tilde{\bm{\zeta}}-\bm{\zeta}\right) \cdot \mathbf{y}} \right| $ in $ \overline{D_{R_1}} $, it follows that
	\begin{align} 
		\label{est_CGOs_zeta_int_g_1-e_varphi}
		& \left\| \bm{\zeta} \int_{D_{R_1}} g_{\tilde{\bm{\zeta}}}\left(\mathbf{x} - \mathbf{y}\right) \left(1 - e^{-\left(\tilde{\bm{\zeta}}-\bm{\zeta}\right) \cdot \mathbf{y}} \right) \varphi \left(\mathbf{y}\right) \mathrm{d} \mathbf{y} \right\|^2_{L^2\left(D_{R_1}\right)} \leqslant \frac{c}{\tau^2} \left\| \varphi \right\|^2_{L^2\left(D_{R_1}\right)} , \\
		\label{est_CGOs_1-ezeta_int_gevarphi}
		& \left\| \left(1 - e^{\left(\tilde{\bm{\zeta}}-\bm{\zeta}\right) \cdot \mathbf{x}} \right) \bm{\zeta} \int_{D_{R_1}} g_{\tilde{\bm{\zeta}}}\left(\mathbf{x} - \mathbf{y}\right) e^{-\left(\tilde{\bm{\zeta}}-\bm{\zeta}\right) \cdot \mathbf{y}} \varphi \left(\mathbf{y}\right) \mathrm{d} \mathbf{y} \right\|^2_{L^2\left(D_{R_1}\right)}  \leqslant \frac{c}{\tau^2} \left\| \varphi \right\|^2_{L^2\left(D_{R_1}\right)} .
	\end{align}
	Then we arrive at  
	\begin{align} \label{est_CGOs_T1}
		\left\| \mathcal{T}_1\left(\varphi\right) \right\|_{L^2\left(D_{R_1}\right)} \leqslant \frac{c}{\tau} \left\| \varphi \right\|_{L^2\left(D_{R_1}\right)}.
	\end{align}	
	An analogous argument for $ \mathcal{T}_1 $ is used to deal with $ \mathcal{T}_2 $. We also split the integrand of $ \mathcal{T}_2\left(\varphi\right) $ into \eqref{eq_CGOs_split_g-eg}. Similar to \eqref{eq_CGOs_zeta_g-g_varphi}, by virtue of \eqref{ineq_CGOs__alpha^2_alpha^4/|alphaalpha-xialpha|^2}, it can be calculated that
	\begin{align} \label{est_CGOs_nabla_int_g-g}
		\left\| \nabla \int_{D_{R_1}} \left(g_{\bm{\zeta}}\left(\mathbf{x} - \mathbf{y}\right) - g_{\tilde{\bm{\zeta}}}\left(\mathbf{x} - \mathbf{y}\right)\right) \varphi \left(\mathbf{y}\right) \mathrm{d} \mathbf{y} \right\|^2_{L^2\left(D_{R_1}\right)} \leqslant \frac{c}{\tau^2} \left\| \varphi \right\|^2_{L^2\left(D_{R_1}\right)} .
	\end{align}
	Using \eqref{est_CGOs_gradient_integral_gf_L^2} and \eqref{ineq_CGOs_sup_1-e}, one derives that 
	\begin{align} \label{est_CGOs_nabla_int_g1-e}
		\left\| \nabla \int_{D_{R_1}} g_{\tilde{\bm{\zeta}}}\left(\mathbf{x} - \mathbf{y}\right) \left(1 - e^{-\left(\tilde{\bm{\zeta}}-\bm{\zeta}\right) \cdot \mathbf{y}} \right) \varphi \left(\mathbf{y}\right) \mathrm{d} \mathbf{y} \right\|^2_{L^2\left(D_{R_1}\right)} \leqslant \frac{c}{\tau^2} \left\| \varphi \right\|^2_{L^2\left(D_{R_1}\right)} .
	\end{align}
	With a strategy similar to \eqref{est_CGOs_zeta_int_g_1-e_varphi} and \eqref{est_CGOs_1-ezeta_int_gevarphi}, we get
	\begin{align}
		\label{est_CGOs_zeta-zetae_int_ge}
		& \left\| (\tilde{\bm{\zeta}}-\bm{\zeta})  e^{\left(\tilde{\bm{\zeta}}-\bm{\zeta}\right) \cdot \mathbf{x}} \int_{D_{R_1}} g_{\tilde{\bm{\zeta}}}\left(\mathbf{x} - \mathbf{y}\right) e^{-\left(\tilde{\bm{\zeta}}-\bm{\zeta}\right) \cdot \mathbf{y}} \varphi \left(\mathbf{y}\right) \mathrm{d} \mathbf{y} \right\|^2_{L^2\left(D_{R_1}\right)} \leqslant \frac{c}{\tau^4} \left\| \varphi \right\|^2_{L^2\left(D_{R_1}\right)}, \\
		\label{est_CGOs_1-e_nabla_int_ge}
		& \left\| \left(1 - e^{\left(\tilde{\bm{\zeta}}-\bm{\zeta}\right) \cdot \mathbf{x}} \right) \nabla \int_{D_{R_1}} g_{\tilde{\bm{\zeta}}}\left(\mathbf{x} - \mathbf{y}\right) e^{-\left(\tilde{\bm{\zeta}}-\bm{\zeta}\right) \cdot \mathbf{y}} \varphi \left(\mathbf{y}\right) \mathrm{d} \mathbf{y} \right\|^2_{L^2\left(D_{R_1}\right)} \leqslant \frac{c}{\tau^2} \left\| \varphi \right\|^2_{L^2\left(D_{R_1}\right)}.
	\end{align}
	Then it yields that 
	\begin{align} \label{est_CGOs_T2} 
		\left\| \mathcal{T}_2\left(\varphi\right) \right\|_{L^2\left(D_{R_1}\right)} \leqslant \frac{c}{\tau} \left\| \varphi \right\|_{L^2\left(D_{R_1}\right)}.
	\end{align}
	By virtue of \eqref{est_CGOs_integral_gf_L^2} and the uniform boundedness of $ \left| e^{-\left(\tilde{\bm{\zeta}}-\bm{\zeta}\right) \cdot \mathbf{y}} \right| $ in $ \overline{D_{R_1}} $, we obtain
	\begin{align} \label{est_CGOs_T3} 
		\left\| \mathcal{T}_3\left(\varphi\right) \right\|_{L^2\left(D_{R_1}\right)} \leqslant \frac{c}{\tau} \left\| \varphi \right\|_{L^2\left(D_{R_1}\right)}.
	\end{align}
	Due to the $ C^2 $-smoothness of $ \rho\left(\mathbf{x}\right) $ in $ \overline{D_{R_1}} $, with the help of \eqref{est_CGOs_integral_gf_L^2}, \eqref{est_CGOs_T1}, \eqref{est_CGOs_T2} and \eqref{est_CGOs_T3}, we derive that $ \mathcal{S} $ is a bounded operator and $ \left\| \mathcal{T} \right\|_{L^2\left(D_{R_1}\right)} \leqslant c/\tau $, where $ c $ is a constant independent of $ \tau $. When $ \tau $ is large enough, we know that $ \left\| \mathcal{T} \right\|_{L^2\left(D_{R_1}\right)} $ is sufficiently small. Then it shows that the operator $ \mathcal{I} - \mathcal{T} $ in \eqref{eq_CGOs_Rv=S+T} is invertible and there exist solutions $ \mathbf{R}\left(\mathbf{x}\right) $ and $ v\left(\mathbf{x}\right) $ to the system \eqref{eq_CGOs_Rv=S+T}. Consequently, the Navier equation \eqref{eq_CGOs_Navier_equation_no_source} has a solution $ \mathbf{u}_0\left(\mathbf{x}\right) $ formulated as \eqref{def_CGOs_u0}. 
	
	For $ \mathbf{R}\left(\mathbf{x}\right) $ and $ v\left(\mathbf{x}\right) $ shown in \eqref{eq_CGOs_Rv=S+T}, we prove estimates about $ \mathbf{R}\left(\mathbf{x}\right) $ and $ v\left(\mathbf{x}\right) $ in $ L^2 $-norm first, and then derive estimates about $ \nabla\mathbf{R}\left(\mathbf{x}\right) $, $ \nabla v\left(\mathbf{x}\right) $ and $ \nabla^2\mathbf{R}\left(\mathbf{x}\right) $ in turn.

	Using \eqref{est_CGOs_integral_gf_L^2}, \eqref{est_CGOs_T1}, \eqref{est_CGOs_T2} and the $ C^2 $-smoothness of $ \rho\left(\mathbf{x}\right) $ in $ \overline{D_{R_1}} $, we have
	\begin{align} \label{ineq_CGOs_RvetaR}
		\left\| \mathbf{R} \right\|_{L^2\left(D_{R_1}\right)} \leqslant \frac{c}{\tau} \left\| v \right\|_{L^2\left(D_{R_1}\right)} + \frac{c}{\tau} \left\| \bm{\eta} \right\|_{L^2\left(D_{R_1}\right)} + \frac{c}{\tau} \left\| \mathbf{R} \right\|_{L^2\left(D_{R_1}\right)}.
	\end{align}
	By virtue of \eqref{est_CGOs_T3}, it yields that
	\begin{align} \label{ineq_CGOs_vvetaR}
		\left\| v \right\|_{L^2\left(D_{R_1}\right)} \leqslant \frac{c}{\tau} \left\| v \right\|_{L^2\left(D_{R_1}\right)} + \frac{c}{\tau} \left\| \bm{\eta} \right\|_{L^2\left(D_{R_1}\right)} + \frac{c}{\tau} \left\| \mathbf{R} \right\|_{L^2\left(D_{R_1}\right)}.
	\end{align} 
	As a result, when $ \tau $ is large enough, combining \eqref{ineq_CGOs_RvetaR} and \eqref{ineq_CGOs_vvetaR}, we arrive at 
	\begin{align} \label{est_CGOs_R_v_L^2}
		\left\| \mathbf{R}\left(\mathbf{x}\right) \right\|_{L^2\left(D_{R_1}\right)} \leqslant \frac{c}{\tau} \left\| \bm{\eta} \right\|_{L^2\left(D_{R_1}\right)}, \quad \left\| v\left(\mathbf{x}\right) \right\|_{L^2\left(D_{R_1}\right)} \leqslant \frac{c}{\tau} \left\| \bm{\eta} \right\|_{L^2\left(D_{R_1}\right)}.
	\end{align}
	
	We now derive the estimate about $ \nabla \mathbf{R}\left(\mathbf{x}\right) $ in $ L^2 $-norm. For $ \varphi \in L^2 \left(D_{R_1}\right) $, we set the operator $ \mathcal{T}_4 $ and $ \mathcal{T}_5 : L^2\left(D_{R_1}\right)\to L^2\left(D_{R_1}\right) $ as
	\begin{align*}
		&\mathcal{T}_4\left(\varphi\right)\left(\mathbf{x}\right) := - \nabla \bm{\zeta} \int_{D_{R_1}} \left(g_{\bm{\zeta}}\left(\mathbf{x} - \mathbf{y}\right) - e^{\left(\tilde{\bm{\zeta}}-\bm{\zeta}\right) \cdot \left(\mathbf{x} - \mathbf{y}\right)} g_{\tilde{\bm{\zeta}}}\left(\mathbf{x} - \mathbf{y}\right)\right) \varphi\left(\mathbf{y}\right) \mathrm{d} \mathbf{y}, \\
		&\mathcal{T}_5\left(\varphi\right)\left(\mathbf{x}\right) := - \nabla^2 \int_{D_{R_1}} \left(g_{\bm{\zeta}}\left(\mathbf{x} - \mathbf{y}\right) - e^{\left(\tilde{\bm{\zeta}}-\bm{\zeta}\right) \cdot \left(\mathbf{x} - \mathbf{y}\right)} g_{\tilde{\bm{\zeta}}}\left(\mathbf{x} - \mathbf{y}\right)\right) \varphi\left(\mathbf{y}\right) \mathrm{d} \mathbf{y}.
	\end{align*}
	Taking gradient at both sides of \eqref{eq_CGOs_Rv=S+T}, one has 
	\begin{align*}
		\nabla\mathbf{R} = -\frac{\omega^2}{\mu} \nabla\mathcal{P}\left(\left(1-\rho\right) \left(\bm{\eta}+\mathbf{R}\right)\right) + \left(\mathcal{T}_4 + \mathcal{T}_5\right) \left(\left(1-\rho\right) v - \nabla \rho \cdot \left(\bm{\eta}+\mathbf{R}\right)\right),
	\end{align*}
	Utilizing \eqref{est_CGOs_gradient_integral_gf_L^2}, it follows that 
	\begin{align*} 
		\left\| \nabla\mathcal{P}\left(\left(1-\rho\right) \left(\bm{\eta}+\mathbf{R}\right)\right) \right\|_{L^2\left(D_{R_1}\right)} \leqslant c \left\| \left(1-\rho\right) \left(\bm{\eta}+\mathbf{R}\right) \right\|_{L^2\left(D_{R_1}\right)} ,
	\end{align*}
	where $ \tau $ is sufficiently large and $ c $ is independent of $ \tau $. We split the integrand of $ \mathcal{T}_4\left(\varphi\right) $ into \eqref{eq_CGOs_split_g-eg}. An analogous reasoning as in \eqref{eq_CGOs_zeta_g-g_varphi} contributes to
	\begin{align} \label{eq_CGOs_gradzeta_g-g_varphi}
		\left\| \nabla \bm{\zeta} \int_{D_{R_1}} \left(g_{\bm{\zeta}}\left(\mathbf{x} - \mathbf{y}\right) - g_{\tilde{\bm{\zeta}}}\left(\mathbf{x} - \mathbf{y}\right)\right) \varphi \left(\mathbf{y}\right) \mathrm{d} \mathbf{y} \right\|^2_{L^2\left(D_{R_1}\right)} \leqslant c \left\| \varphi \right\|^2_{L^2\left(D_{R_1}\right)},
	\end{align}
	where $ \tau $ is large enough and $ c $ is a constant independent of $ \tau $. Applying \eqref{est_CGOs_gradient_integral_gf_L^2} and the similar strategy in \eqref{est_CGOs_zeta_int_g_1-e_varphi} and \eqref{est_CGOs_1-ezeta_int_gevarphi}, it can be calculated that
	\begin{align*} 
		& \left\| \nabla \bm{\zeta} \int_{D_{R_1}} g_{\tilde{\bm{\zeta}}}\left(\mathbf{x} - \mathbf{y}\right) \left(1 - e^{-\left(\tilde{\bm{\zeta}}-\bm{\zeta}\right) \cdot \mathbf{y}} \right) \varphi \left(\mathbf{y}\right) \mathrm{d} \mathbf{y} \right\|^2_{L^2\left(D_{R_1}\right)} \leqslant c \left\| \varphi \right\|^2_{L^2\left(D_{R_1}\right)} , \\
		& \left\| \left(1 - e^{\left(\tilde{\bm{\zeta}}-\bm{\zeta}\right) \cdot \mathbf{x}} \right) \nabla \bm{\zeta} \int_{D_{R_1}} g_{\tilde{\bm{\zeta}}}\left(\mathbf{x} - \mathbf{y}\right) e^{-\left(\tilde{\bm{\zeta}}-\bm{\zeta}\right) \cdot \mathbf{y}} \varphi \left(\mathbf{y}\right) \mathrm{d} \mathbf{y} \right\|^2_{L^2\left(D_{R_1}\right)} \leqslant c \left\| \varphi \right\|^2_{L^2\left(D_{R_1}\right)} .
	\end{align*}
	By virtue of \eqref{est_CGOs_integral_gf_L^2} and the analogous reasoning as above, one can show that
	\begin{align*} 
		\left\| \bm{\zeta} \left(\tilde{\bm{\zeta}}-\bm{\zeta}\right)^{\top} e^{\left(\tilde{\bm{\zeta}}-\bm{\zeta}\right) \cdot \mathbf{x}} \int_{D_{R_1}} g_{\tilde{\bm{\zeta}}}\left(\mathbf{x} - \mathbf{y}\right) e^{-\left(\tilde{\bm{\zeta}}-\bm{\zeta}\right) \cdot \mathbf{y}} \varphi \left(\mathbf{y}\right) \mathrm{d} \mathbf{y} \right\|^2_{L^2\left(D_{R_1}\right)} \leqslant \frac{c}{\tau^2} \left\| \varphi \right\|^2_{L^2\left(D_{R_1}\right)} .
	\end{align*}
	Then we deduce that 
	\begin{align*}
		\left\| \mathcal{T}_4 \left(\left(1-\rho\right) v - \nabla \rho \cdot \left(\bm{\eta}+\mathbf{R}\right)\right) \right\|_{L^2\left(D_{R_1}\right)} \leqslant c \left\| \left(1-\rho\right) v - \nabla \rho \cdot \left(\bm{\eta}+\mathbf{R}\right) \right\|_{L^2\left(D_{R_1}\right)} .
	\end{align*}	
	We estimate $ \mathcal{T}_5 $ by means of the analogous method. Split the integrand of $ \mathcal{T}_5 $ into \eqref{eq_CGOs_split_g-eg}. Notice that $ \left|\bm{\alpha}\right| \geqslant \left|\alpha_1\right| \geqslant \frac{\pi}{2R^{\prime}} $. From \eqref{ineq_CGOs__alpha^2_alpha^4/|alphaalpha-xialpha|^2}, we have
	\begin{align*} 
		\frac{\left| \bm{\alpha} \right|^3}{\left|\bm{\alpha} \cdot \bm{\alpha} - 2\mathrm{i} \bm{\xi} \cdot \bm{\alpha}\right|^2} \leqslant \frac{2R^{\prime}}{\pi}\frac{\left| \bm{\alpha} \right|^4}{\left|\bm{\alpha} \cdot \bm{\alpha} - 2\mathrm{i} \bm{\xi} \cdot \bm{\alpha}\right|^2} \leqslant c_{k,R^{\prime}} \tau^2.
	\end{align*}
	Similar to \eqref{eq_CGOs_gradzeta_g-g_varphi}, it follows that
	\begin{align*}
		\left\| \nabla^2 \int_{D_{R_1}} \left(g_{\bm{\zeta}}\left(\mathbf{x} - \mathbf{y}\right) - g_{\tilde{\bm{\zeta}}}\left(\mathbf{x} - \mathbf{y}\right)\right) \varphi \left(\mathbf{y}\right) \mathrm{d} \mathbf{y} \right\|^2_{L^2\left(D_{R_1}\right)} \leqslant c \left\| \varphi \right\|^2_{L^2\left(D_{R_1}\right)} .
	\end{align*}
	Using \eqref{est_CGOs_gradient_gradient_integral_gf_L^2}, \eqref{ineq_CGOs_sup_1-e} and the uniform boundedness of $ \left| e^{-\left(\tilde{\bm{\zeta}}-\bm{\zeta}\right) \cdot \mathbf{x}} \right| $, it arrives at 
	\begin{align*} 
		& \left\| \nabla^2 \int_{D_{R_1}} g_{\tilde{\bm{\zeta}}}\left(\mathbf{x} - \mathbf{y}\right) \left(1 - e^{-\left(\tilde{\bm{\zeta}}-\bm{\zeta}\right) \cdot \mathbf{y}} \right) \varphi \left(\mathbf{y}\right) \mathrm{d} \mathbf{y} \right\|^2_{L^2\left(D_{R_1}\right)} \leqslant c \left\| \varphi \right\|^2_{L^2\left(D_{R_1}\right)} . \\
		& \left\| \left(1 - e^{\left(\tilde{\bm{\zeta}}-\bm{\zeta}\right) \cdot \mathbf{x}} \right) \nabla^2 \int_{D_{R_1}} g_{\tilde{\bm{\zeta}}}\left(\mathbf{x} - \mathbf{y}\right) e^{-\left(\tilde{\bm{\zeta}}-\bm{\zeta}\right) \cdot \mathbf{y}} \varphi \left(\mathbf{y}\right) \mathrm{d} \mathbf{y} \right\|^2_{L^2\left(D_{R_1}\right)} \leqslant c \left\| \varphi \right\|^2_{L^2\left(D_{R_1}\right)} .
	\end{align*}
	By virtue of \eqref{est_CGOs_gradient_integral_gf_L^2} and \eqref{ineq_CGOs_tildexi-xi}, the analogous strategy in \eqref{est_CGOs_zeta_int_g_1-e_varphi} and \eqref{est_CGOs_1-ezeta_int_gevarphi} leads to
	\begin{align*} 
		& \left\| (\tilde{\bm{\zeta}}-\bm{\zeta}) (\tilde{\bm{\zeta}}-\bm{\zeta})^{\top} e^{\left(\tilde{\bm{\zeta}}-\bm{\zeta}\right) \cdot \mathbf{x}} \int_{D_{R_1}} g_{\tilde{\bm{\zeta}}}\left(\mathbf{x} - \mathbf{y}\right) e^{-\left(\tilde{\bm{\zeta}}-\bm{\zeta}\right) \cdot \mathbf{y}} \varphi \left(\mathbf{y}\right) \mathrm{d} \mathbf{y} \right\|^2_{L^2\left(D_{R_1}\right)} \notag \leqslant \frac{c}{\tau^6} \left\| \varphi \right\|^2_{L^2\left(D_{R_1}\right)}, \\
		& \left\| e^{\left(\tilde{\bm{\zeta}}-\bm{\zeta}\right) \cdot \mathbf{x}} \nabla \int_{D_{R_1}} g_{\tilde{\bm{\zeta}}}\left(\mathbf{x} - \mathbf{y}\right) e^{-\left(\tilde{\bm{\zeta}}-\bm{\zeta}\right) \cdot \mathbf{y}} \varphi \left(\mathbf{y}\right) \mathrm{d} \mathbf{y} \left(\tilde{\bm{\zeta}}-\bm{\zeta}\right)^{\top} \right\|^2_{L^2\left(D_{R_1}\right)} \leqslant \frac{c}{\tau^2} \left\| \varphi \right\|^2_{L^2\left(D_{R_1}\right)} .
	\end{align*}
	From the above estimates, it yields that 
	\begin{align} \label{est_CGOs_I6} 
		\left\| \mathcal{T}_5 \left(\left(1-\rho\right) v - \nabla \rho \cdot \left(\bm{\eta}+\mathbf{R}\right)\right) \right\|_{L^2\left(D_{R_1}\right)} \leqslant c \left\| \left(1-\rho\right) v - \nabla \rho \cdot \left(\bm{\eta}+\mathbf{R}\right) \right\|_{L^2\left(D_{R_1}\right)} .
	\end{align}
	With the help of \eqref{est_CGOs_R_v_L^2}, we finally get 
	\begin{align} \label{est_CGOs_gradientR_L^2}
		\left\| \nabla\mathbf{R}\left(\mathbf{x}\right) \right\|_{L^2\left(D_{R_1}\right)} \leqslant c \left\| \bm{\eta} \right\|_{L^2\left(D_{R_1}\right)}.
	\end{align}

	The estimate about $ \nabla v\left(\mathbf{x}\right) $ is demonstrated below. We define the operator $ \mathcal{T}_6 $ and $ \mathcal{T}_7 : L^2\left(D_{R_1}\right)\to L^2\left(D_{R_1}\right) $ as
	\begin{align*}
		&\mathcal{T}_6\left(\varphi\right)\left(\mathbf{x}\right) := - k_{\mathrm{p}}^2 \left(\tilde{\bm{\zeta}}-\bm{\zeta}\right) e^{\left(\tilde{\bm{\zeta}}-\bm{\zeta}\right) \cdot \mathbf{x}} \int_{D_R} g_{\tilde{\bm{\zeta}}}\left(\mathbf{x} - \mathbf{y}\right) e^{-\left(\tilde{\bm{\zeta}}-\bm{\zeta}\right) \cdot \mathbf{y}} \varphi\left(\mathbf{y}\right) \mathrm{d} \mathbf{y}, \\
		&\mathcal{T}_7\left(\varphi\right)\left(\mathbf{x}\right) := - k_{\mathrm{p}}^2 e^{\left(\tilde{\bm{\zeta}}-\bm{\zeta}\right) \cdot \mathbf{x}} \nabla \int_{D_R} g_{\tilde{\bm{\zeta}}}\left(\mathbf{x} - \mathbf{y}\right) e^{-\left(\tilde{\bm{\zeta}}-\bm{\zeta}\right) \cdot \mathbf{y}} \varphi\left(\mathbf{y}\right) \mathrm{d} \mathbf{y},
	\end{align*}
	where $ \varphi \in L^2 \left(D_{R_1}\right) $. By \eqref{eq_CGOs_Rv=S+T}, it is easy to know that 
	\begin{align*}
		\nabla v\left(\mathbf{x}\right) = \left(\mathcal{T}_6 + \mathcal{T}_7\right) \left(\left(1-\rho\right) v - \nabla \rho \cdot \left(\bm{\eta}+\mathbf{R}\right)\right).
	\end{align*}
	Using the uniform boundedness of $ \left| e^{-\left(\tilde{\bm{\zeta}}-\bm{\zeta}\right) \cdot \mathbf{x}} \right| $ in $ \overline{D_{R_1}} $, \eqref{ineq_CGOs_tildexi-xi} and \eqref{est_CGOs_integral_gf_L^2}, we have
	\begin{align*} 
		\left\| \mathcal{T}_6 \left(\left(1-\rho\right) v - \nabla \rho \cdot \left(\bm{\eta}+\mathbf{R}\right)\right) \right\|_{L^2\left(D_{R_1}\right)} \leqslant \frac{c}{\tau^2} \left\| \left(1-\rho\right) v - \nabla \rho \cdot \left(\bm{\eta}+\mathbf{R}\right) \right\|_{L^2\left(D_{R_1}\right)} .
	\end{align*}
	Similarly, it arrives at
	\begin{align*} 
		\left\| \mathcal{T}_7 \left(\left(1-\rho\right) v - \nabla \rho \cdot \left(\bm{\eta}+\mathbf{R}\right)\right) \right\|_{L^2\left(D_{R_1}\right)} \leqslant c \left\| \left(1-\rho\right) v - \nabla \rho \cdot \left(\bm{\eta}+\mathbf{R}\right) \right\|_{L^2\left(D_{R_1}\right)} .
	\end{align*}
	By virtue of \eqref{est_CGOs_R_v_L^2}, we obtain 
	\begin{align} \label{est_CGOs_gradientv_L^2}
		\left\| \nabla v\left(\mathbf{x}\right) \right\|_{L^2\left(D_{R_1}\right)} \leqslant c \left\| \bm{\eta} \right\|_{L^2\left(D_{R_1}\right)}.
	\end{align}
	
	Now we consider the estimate about $ \nabla^2 \mathbf{R}\left(\mathbf{x}\right) $. For $ \varphi \in L^2 \left(D_{R_1}\right) $, let the operator $ \mathcal{T}_8 $ and $ \mathcal{T}_9 : L^2\left(D_{R_1}\right)\to L^2\left(D_{R_1}\right) $ be defined as
	\begin{align*}
		&\mathcal{T}_8\left(\varphi\right)\left(\mathbf{x}\right) := - \nabla^2 \bm{\zeta} \int_{D_{R_1}} \left(g_{\bm{\zeta}}\left(\mathbf{x} - \mathbf{y}\right) - e^{\left(\tilde{\bm{\zeta}}-\bm{\zeta}\right) \cdot \left(\mathbf{x} - \mathbf{y}\right)} g_{\tilde{\bm{\zeta}}}\left(\mathbf{x} - \mathbf{y}\right)\right) \varphi\left(\mathbf{y}\right) \mathrm{d} \mathbf{y}, \\
		&\mathcal{T}_9\left(\varphi\right)\left(\mathbf{x}\right) := - \nabla^3 \int_{D_{R_1}} \left(g_{\bm{\zeta}}\left(\mathbf{x} - \mathbf{y}\right) - e^{\left(\tilde{\bm{\zeta}}-\bm{\zeta}\right) \cdot \left(\mathbf{x} - \mathbf{y}\right)} g_{\tilde{\bm{\zeta}}}\left(\mathbf{x} - \mathbf{y}\right)\right) \varphi\left(\mathbf{y}\right) \mathrm{d} \mathbf{y}.
	\end{align*}
	From \eqref{eq_CGOs_Rv=S+T}, we know that 
	\begin{align*}
		\nabla^2\mathbf{R} = -\frac{\omega^2}{\mu} \nabla^2\mathcal{P}\left(\left(1-\rho\right) \left(\bm{\eta}+\mathbf{R}\right)\right) + \left(\mathcal{T}_8 + \mathcal{T}_9\right) \left(\left(1-\rho\right) v - \nabla \rho \cdot \left(\bm{\eta}+\mathbf{R}\right)\right),
	\end{align*}
	Using \eqref{est_CGOs_gradient_gradient_integral_gf_L^2}, one has
	\begin{align*} 
		\left\| \nabla^2\mathcal{P}\left(\left(1-\rho\right) \left(\bm{\eta}+\mathbf{R}\right)\right) \right\|_{L^2\left(D_{R_1}\right)} \leqslant c \tau \left\| \left(1-\rho\right) \left(\bm{\eta}+\mathbf{R}\right) \right\|_{L^2\left(D_{R_1}\right)} .
	\end{align*}
	Similar to the proof of \eqref{est_CGOs_I6}, with the help of \eqref{est_CGOs_gradient_gradient_integral_gf_L^2}, we readily derive that
	\begin{align*} 
		\left\| \mathcal{T}_8 \left(\left(1-\rho\right) v - \nabla \rho \cdot \left(\bm{\eta}+\mathbf{R}\right)\right) \right\|_{L^2\left(D_{R_1}\right)} &\leqslant c \tau \left\| \mathcal{T}_5 \left(\left(1-\rho\right) v - \nabla \rho \cdot \left(\bm{\eta}+\mathbf{R}\right)\right) \right\|_{L^2\left(D_{R_1}\right)} \\
		&\leqslant c \tau \left\| \left(1-\rho\right) v - \nabla \rho \cdot \left(\bm{\eta}+\mathbf{R}\right) \right\|_{L^2\left(D_{R_1}\right)} .
	\end{align*}
	Regarding to $ \mathcal{T}_9\left(\varphi\right) $, for all $ \varphi \in C^{\infty}_0 \left(D_{R_1}\right) $, using the formula for integration by parts of a double integral, we have
	\begin{align*} 
		\mathcal{T}_9\left(\varphi\right) \left(\mathbf{x}\right) = \nabla^2 \int_{D_{R_1}} \left(g_{\bm{\zeta}}\left(\mathbf{x} - \mathbf{y}\right) - e^{\left(\tilde{\bm{\zeta}}-\bm{\zeta}\right) \cdot \left(\mathbf{x} - \mathbf{y}\right)} g_{\tilde{\bm{\zeta}}}\left(\mathbf{x} - \mathbf{y}\right)\right) \nabla \varphi \left(\mathbf{y}\right) \mathrm{d} \mathbf{y}.
	\end{align*}
	From \eqref{est_CGOs_I6}, it follows that $ \left\| \mathcal{T}_9\left(\varphi\right) \right\|_{L^2\left(D_{R_1}\right)} \leqslant c \left\| \nabla \varphi \right\|_{L^2\left(D_{R_1}\right)} $, which contributes to
	\begin{align*} 
		\left\| \mathcal{T}_9 \left(\left(1-\rho\right) v - \nabla \rho \cdot \left(\bm{\eta}+\mathbf{R}\right)\right) \right\|_{L^2\left(D_{R_1}\right)} \leqslant c \left\| \nabla \left(\left(1-\rho\right) v - \nabla \rho \cdot \left(\bm{\eta}+\mathbf{R}\right)\right) \right\|_{L^2\left(D_{R_1}\right)} .
	\end{align*}
	By virtual of \eqref{est_CGOs_R_v_L^2}, \eqref{est_CGOs_gradientR_L^2} and \eqref{est_CGOs_gradientv_L^2}, we finally obtain 
	\begin{align*}
		\left\| \nabla^2\mathbf{R}\left(\mathbf{x}\right) \right\|_{L^2\left(D_{R_1}\right)} \leqslant c\tau \left\| \bm{\eta} \right\|_{L^2\left(D_{R_1}\right)}.
	\end{align*}	

	The proof is complete.
\end{proof}

\section*{Acknowledgments}

The work of H. Diao is supported by National Natural Science Foundation of China (No. 12371422) and the Fundamental Research Funds for the Central Universities, JLU.


\end{document}